 \documentclass[a4paper,11pt]{article}
 \usepackage[plainpages=false]{hyperref}
 \usepackage{amsfonts,amsmath,amssymb,amsthm}
 \usepackage{latexsym,lscape,rawfonts,mathrsfs}

 \usepackage[dvips]{color}
 \usepackage{multicol}

 \usepackage[all]{xy}
 \usepackage{eufrak}
 \usepackage{makeidx}
 \usepackage{graphicx,psfrag}

 \usepackage{array,tabularx}

 \usepackage{setspace}

 \usepackage{appendix}

\usepackage{txfonts}

 \newcommand{\C}{\ensuremath{\mathbb{C}}}
 \newcommand{\D}[2]{\ensuremath{ \frac{\partial{#1}}{\partial{#2}}}}

 \newcommand{\Q}{\ensuremath{\mathbb{Q}}}
 \newcommand{\R}{\ensuremath{\mathbb{R}}}
 \newcommand{\Z}{\ensuremath{\mathbb{Z}}}

 \newcommand{\ba}{\begin{align*}}
 \newcommand{\ea}{\end{align*}}
\def\Fut{{\mathrm{Fut}}}

 \newcommand{\norm}[2]{{ \ensuremath{\left\|} #1 \ensuremath{\right\|}}_{#2}}

 \makeatletter
 \def\ExtendSymbol#1#2#3#4#5{\ext@arrow 0099{\arrowfill@#1#2#3}{#4}{#5}}
 
 \makeatother

 \makeatletter
 \def\ExtendSymbol#1#2#3#4#5{\ext@arrow 0099{\arrowfill@#1#2#3}{#4}{#5}}
 \newcommand\longright[2][]{\ExtendSymbol{-}{-}{\rightarrow}{#1}{#2}}

 \definecolor{hao}{rgb}{1,0.5,0}
 \definecolor{miao}{cmyk}{0.5,0,0.2,0.2}
 \definecolor{qiao}{gray}{0.96}

\numberwithin{equation}{section}

\newtheorem{prop}{Proposition}[section]

\newtheorem{proposition}[prop]{Proposition}
\newtheorem{theoremin}[prop]{Theorem}
\newtheorem{theorem}[prop]{Theorem}

\newtheorem{lemma}[prop]{Lemma}
\newtheorem{claim}[prop]{Claim}

\newtheorem{corollary}[prop]{Corollary}

\newtheorem{remark}[prop]{Remark}
\newtheorem{example}[prop]{Example}

\newtheorem{definition}[prop]{Definition}
\newtheorem{conjecturein}[prop]{Conjecture}

 \newenvironment{dedication}
        {\begin{quotation}\begin{center}\begin{em}}
        {\par\end{em}\end{center}\end{quotation}}

\newcommand{\p}{\partial}
\newcommand{\vphi}{\varphi}
\newcommand{\om}{\omega}

\newcommand{\pd}[2]{\frac {\partial #1}{\partial #2}}

\newcommand{\al}{\alpha}

\newcommand{\la}{\lambda}
\newcommand{\La}{\Lambda}
\newcommand{\oo}{\omega}

\newcommand{\dd}{\delta}
\newcommand{\Na}{\nabla}

\newcommand{\ee}{\epsilon}

\def\Rm{{\rm Rm}}

\newcommand{\beq}{\begin{equation}}
\newcommand{\eeq}{\end{equation}}
\newcommand{\beqs}{\begin{eqnarray*}}
\newcommand{\eeqs}{\end{eqnarray*}}
\newcommand{\beqn}{\begin{eqnarray}}
\newcommand{\eeqn}{\end{eqnarray}}
\newcommand{\beqa}{\begin{array}}
\newcommand{\eeqa}{\end{array}}

\def\td{\tilde}
\def\p{\partial}

\def\ZZ{{\mathbb Z}}

\def\RR{{\mathbb R}}

\def\PP{{\mathbb P}}
\def\CC{{\mathbb C}}

\def\ri{\rightarrow}

\def\un{\underline}

\def\no{{\nonumber}}

\def\vol{{\rm Vol}}

\def\ba{{\mathbf a}}

\def\cF{{\mathcal F}}

\def\ba{\Xi}

\title{Regularity scales and convergence of  the Calabi flow}
\author{Haozhao Li \footnote{Supported by NSFC grant No. 11131007.}, Bing Wang \footnote{Supported by NSF grant DMS-1312836.}, and Kai Zheng}
\date{}

\begin{document}
\maketitle

\begin{dedication}
In memory of Professor Weiyue Ding
\end{dedication}

\begin{abstract}
  We define regularity scales as alternative quantities of $\displaystyle \left(\max_{M} |Rm| \right)^{-1}$ to study the behavior of the Calabi flow.
Based on estimates of the regularity scales, we obtain convergence  theorems of the Calabi flow on extremal K\"ahler surfaces,
under the assumption of global existence of the Calabi flow solutions.
Our results partially confirm Donaldson's conjectural picture for the Calabi flow in complex dimension 2.
Similar results hold in high dimension with an extra assumption that the scalar curvature is uniformly bounded.
\end{abstract}

\tableofcontents

\section{Introduction}
In the seminal work~\cite{[Cal1]},  E. Calabi studied the variational problem of
the functional $\int_M (S-\underline{S})^2$, the Calabi energy,
among K\"ahler metrics in a fixed cohomology class.
The vanishing points of the Calabi energy are called the constant scalar curvature K\"ahler (cscK) metrics.
The critical points of the Calabi energy are called the extremal K\"ahler (extK) metrics.
To search such metrics, Calabi introduced a geometric flow, which is now well known as the Calabi flow.
Actually,  on a K\"ahler manifold $(M^n, \omega, J)$, the Calabi flow deforms the metric by
\begin{align}
\frac{\partial}{\partial t} {g_{i\bar{j}}}=S_{, i\bar j}, \label{eq000}
\end{align}
where $g$ is the metric determined by $\omega(t)$ and $J$, and $S$ is the scalar curvature of $g$.
Note that in the class $[\omega]$, every metric form can be written as
$\omega+\sqrt{-1} \partial \bar{\partial} \varphi$ for some smooth K\"ahler potential function $\varphi$.
Therefore, on the K\"ahler potential level, the above equation reduces to
\begin{align}
   \frac{\partial}{\partial t} \varphi= S-\underline{S}
   =-g^{i\bar{j}} \left\{ \log \det \left( g_{k\bar{l}} + \varphi_{k\bar{l}}\right) \right\}_{,i\bar{j}}
     -\underline{S},
\label{eqn:HG29_2}
\end{align}
where $\underline{S}$ is the average of scalar curvature, which is a constant depending only on
the class $[\omega]$.  Note that equation (\ref{eqn:HG29_2}) is a fourth order fully nonlinear PDE.  This order incurs extreme technical difficulty.
In spite of this difficulty, the short time existence of equation (\ref{eqn:HG29_2}) was proved by X.X. Chen and  W.Y. He in~\cite{[ChenHe1]}.
Furthermore, they also proved the global existence of (\ref{eqn:HG29_2}) under the assumption that the Ricci curvature is uniformly bounded.

About two decades after the birth of the Calabi flow, in ~\cite{[Don]}, S.K. Donaldson (See \cite{Fujiki} also) pointed out that the Calabi flow
 fits into a general frame of moment map picture.
In fact, by fixing the underlying symplectic manifold $(M, \omega)$ and deforming the almost complex structures $J$ along
Hamiltonian vector fields, $C^{\infty}(M)$ has an infinitesimal action on the moduli space of almost complex structures.
The function $S-\underline{S}$ can be regarded as the moment map of this action, where $S$ is the Hermitian scalar curvature in general.
Therefore, $\int_M (S-\underline{S})^2$ is the moment map square function, defined on the moduli space of almost complex structures.
Then the downward gradient flow of the moment map norm square can be written as
\begin{align}
      \frac{d}{dt} J= -\frac{1}{2} J \circ \bar{\partial}_{J} X_{S},
\label{eqn:HL31_1}
\end{align}
where $X_S$ is the symplectic dual vector field of $dS$.
When the flow path of (\ref{eqn:HL31_1}) locates in the integrable almost complex structures, the Hermitian scalar coincides with the Riemannian scalar curvature.
Therefore, the flow (\ref{eqn:HL31_1}) is nothing but the classical Calabi flow (\ref{eq000}) up to diffeomorphisms.
Based on this moment map picture, Donaldson then described some conjectural behaviors of the Calabi flow.

\begin{conjecturein}[S.K. Donaldson \cite{[Don]}]
 Suppose the Calabi flows have global existence.
 Then the asymptotic behavior of the Calabi flow starting from $(M, \omega, J)$ falls into one of the four possibilities.
\begin{enumerate}
\item The flow converges to a cscK metric on the same complex manifold $(M,J)$.
\item The flow is asymptotic to a one-parameter family of extK metrics on the same complex manifold $(M,J)$, evolving by diffeomorphisms.
\item The manifold does not admit an extK metric but the transformed flow $J_t$ on $\mathcal{J}$ converges to $J'$.  Furthermore, one can construct a destabilizing test configuration of
$(M, J)$ such that $(M, J')$ is the central fiber.
\item The transformed flow $J_t$ on $\mathcal{J}$ does not converge in smooth topology and singularities develop.
         However, one can still make sufficient sense of the limit of $J_t$ to extract a scheme from it, and this scheme can be fitted in as the central fiber of a destabilizing test configuration.
\end{enumerate}
\label{cjein:J04_2}
\end{conjecturein}

Conjecture~\ref{cjein:J04_2} has attracted a lot of attentions for the study of the Calabi flow.
On the way to understand it,  there are many important works.
For example, R. Berman~\cite{MR3107540}, W.Y. He~\cite{[He5]} and J. Streets~\cite{[Streets1]} proved the convergence
of the Calabi flow in various topologies, under different geometric conditions.
G. Sz{\'e}kelyhidi~\cite{MR2571959}  constructed examples of global solutions of the Calabi flow which collapse at time infinity.
A finite-dimension-approximation approach to study the Calabi flow was developed in~\cite{Fine} by J. Fine.

Note that the global existence of the Calabi flows is a fundamental assumption in Conjecture~\ref{cjein:J04_2}.
On Riemann surfaces, the global existence and the convergence of the Calabi flow have been proved by P.T. Chrusical~\cite{[Chru]},  X.X. Chen~\cite{[Chen]} and M. Struwe~\cite{[Stru]}.
However, much less is known in high dimension. It was conjectured by X.X. Chen~\cite{[Chenconj]} that every Calabi flow has global existence.
This conjecture sounds to be too optimistic at the beginning.  However,  there are positive evidences for it.
In~\cite{MR3197669},  J. Streets  proved the global existence of the minimizing movement flow, which can be regarded as weak Calabi flow solutions.
Therefore, the global existence of the Calabi flow can be proved if one can fully improve the regularity of the minimizing movement flow,  although there exist terrific analytic difficulties to achieve this.
In general,  Chen's conjecture was only confirmed in particular cases.  For example, if the underlying manifold is an Abelian surface and the initial metric is $T$-invariant,  H.N. Huang and R.J. Feng
proved the global existence in~\cite{HuangFeng}.


In short, the Calabi flow can be understood from two points of view: either as a flow of metric forms (Calabi's point of view) within a given cohomologous class on a fixed complex manifold,
or as a flow of complex structures (Donaldson's point of view) on a fixed symplectic manifold.
Let $(M^n, \omega, J)$ be a reference K\"ahler manifold, $g$ be the reference metric determined by $\omega$ and $J$.
If $J$ is fixed, then the Calabi flow evolves in the space
\begin{align}
  \mathcal{H} \triangleq \left\{ \omega_{\varphi}
  \left| \varphi \in C^{\infty}(M), \omega_{\varphi}=\omega + \sqrt{-1}\partial \bar{\partial} \varphi>0 \right.\right\}.
\label{eqn:HL30_1}
\end{align}
If $\omega$ is fixed, then the Calabi flow evolves in  the space
\begin{align}
   \mathcal{J} \triangleq \left\{  J' | J' \;  \textrm{is an integrable almost complex structure compatible with} \; \omega \right\}.
\label{eqn:HL31_2}
\end{align}
We equip both $\mathcal{H}$ and $\mathcal{J}$ with $C^{k,\frac{1}{2}}$ topology for some sufficiently large $k=k(n)$, with respect to the reference metric $g$.
Each point of view of the Calabi flow has its own advantage.
We shall take both points of view and may jump from one to the other without mentioning this explicitly.

\begin{theoremin}
 Suppose $(M^2, \omega, J)$ is an extK surface.
 Define
 \begin{align*}
  &\widetilde{\mathcal{LH}} \triangleq \left\{ \omega_{\varphi} \in \mathcal{H}| \textrm{The Calabi flow initiating from $\omega_{\varphi}$ has global existence} \right\}, \\
  &\mathcal{LH} \triangleq \textrm{The path-connected component of} \; \widetilde{\mathcal{LH}} \; \textrm{containing} \; \omega.
 \end{align*}
 Then the modified Calabi flow (c.f. Definition~\ref{dfn:HL27_1}) starting from any $\omega_{\varphi} \in \mathcal{LH}$ converges to $\varrho^\ast\omega$ for some $\varrho\in \mathrm{Aut}_0(M,J)$,
 in the smooth topology of K\"ahler potentials.
\label{thmin:1}
\end{theoremin}

In the setup of Conjecture~\ref{cjein:J04_2}, we have $\mathcal{LH}=\widetilde{\mathcal{LH}}=\mathcal{H}$ automatically.
Therefore, Theorem~\ref{thmin:1} confirms the first two possibilities of Conjecture~\ref{cjein:J04_2} in complex dimension 2.

\begin{theoremin}
 Suppose $\left\{ (M^2, \omega, J_s), s \in D\right\}$ is a smooth family of K\"ahler surfaces parametrized by the disk $D=\{z| z \in \C, |z|<2\}$,
 with the following conditions satisfied.
 \begin{itemize}
   \item There is a smooth family of diffeomorphisms $\{\psi_s: s \in D \backslash \{0\} \}$ such that
      \begin{align*}
         [\psi_s^* \omega] = [\omega], \quad \psi_s^* J_{s}= J_{1}, \quad \psi_1=Id.
      \end{align*}
 \item  $[\omega]$ is integral.
 \item  $(M^2, \omega, J_0)$ is a cscK surface.
 \end{itemize}
Denote
 \begin{align*}
  &\widetilde{\mathcal{LJ}} \triangleq \left\{ J_{s}| s\in D, \textrm{the Calabi flow initiating from $(M, \omega, J_s)$ has global existence} \right\}, \\
  &\mathcal{LJ} \triangleq \textrm{The path-connected component of} \; \widetilde{\mathcal{LJ}} \; \textrm{containing} \; J_0.
 \end{align*}
 Then the Calabi flow starting from any $J_s \in \mathcal{LJ}$ converges to $\psi^*(J_0)$ in the smooth topology of sections of $TM \otimes T^*M$,
 where $\psi \in Symp(M,\omega)$ depends on $J_s$.
 \label{thmin:2}
\end{theoremin}

Theorem~\ref{thmin:2} partially confirms the third possibility of Conjecture~\ref{cjein:J04_2} in complex dimension 2,  in the case that the $C^{\infty}$-closure of
the $\mathcal{G}^{\C}$-leaf of $J_1$ contains a cscK complex structure, for a polarized K\"ahler surface.
Note that by the integral condition of $[\omega]$ and reductivity of the automorphism groups of cscK complex manifolds,
the construction of destablizing test configurations follows from~\cite{Don10} directly. \\

 Theorem~\ref{thmin:1} and Theorem~\ref{thmin:2} have high dimensional counterparts.
 However,  in high dimension, due to the loss of scaling invariant property of the Calabi energy,
 we need some extra assumptions of scalar curvature to guarantee the convergence.

\begin{theoremin}
 Suppose $(M^n, \omega, J)$ is an extremal K\"ahler manifold.

 For each big constant $A$, we set
 \begin{align*}
  &\widetilde{\mathcal{LH}}_{A} \triangleq \left\{ \omega_{\varphi} \in \mathcal{H} | \textrm{The Calabi flow initiating from $\omega_{\varphi}$ has global existence and $|S|\leq A$} \right\}, \\
  &\mathcal{LH}_{A} \triangleq \textrm{The path-connected component of} \; \widetilde{\mathcal{LH}}_A \; \textrm{containing} \; \omega, \\
  &\mathcal{LH}' \triangleq \bigcup_{A>0} \mathcal{LH}_A.
 \end{align*}
 Then the modified Calabi flow starting from each  $\omega_{\varphi} \in \mathcal{LH}'$ converges to $\varrho^\ast\omega$ for some $\varrho =\varrho(\varphi) \in \mathrm{Aut}_0(M,J)$,
 in the smooth topology of K\"ahler potentials.
 \label{thmin:3}
\end{theoremin}

Note that by Chen-He's stability theorem(c.f.~\cite{[ChenHe1]}), the set $\mathcal{LH}_{A}$ is non-empty if $A$ is large enough.
Therefore, $\mathcal{LH}'$ is a non-empty subset of $\mathcal{LH}$. We have the relationships
\begin{align}
    \mathcal{LH}' \subset \mathcal{LH} \subset \mathcal{H}.    \label{eqn:HG29_1}
\end{align}
Therefore, in order to understand the global behavior of the Calabi flow,  it is crucial to set up the equalities.
\begin{align}
&\mathcal{LH}=\mathcal{H},  \label{eqn:HG29_3} \\
&\mathcal{LH}'=\mathcal{LH}.  \label{eqn:HG29_4}
\end{align}
Equality (\ref{eqn:HG29_3}) is nothing but the restatement of Chen's conjecture.
Equality (\ref{eqn:HG29_4}) is more or less a global scalar curvature bound estimate.

\begin{theoremin}
  Suppose $\left\{ (M^n, \omega, J_s), s \in D \right\}$ is a smooth family of K\"ahler manifolds parametrized by the disk $D=\{z|z \in \C, |z|<2\}$,
 with the following conditions satisfied.
 \begin{itemize}
   \item There is a smooth family of diffeomorphisms $\{\psi_s: s \in D \backslash \{0\} \}$ such that
      \begin{align*}
         [\psi_s^* \omega] = [\omega], \quad \psi_s^* J_{s}= J_{1}, \quad \psi_1=Id.
      \end{align*}
 \item $[\omega]$ is integral.
 \item  $(M^n, \omega, J_0)$ is a cscK manifold.
 \end{itemize}
Denote
\begin{align*}
  &\widetilde{\mathcal{LJ}}_{A} \triangleq \left\{ J_s| s \in D, \textrm{the Calabi flow initiating from $J_s$ has global existence and $|S|\leq A$} \right\}, \\
  &\mathcal{LJ}_{A} \triangleq \textrm{The path-connected component of} \; \widetilde{\mathcal{LJ}}_A \; \textrm{containing} \; J_0, \\
  &\mathcal{LJ}' \triangleq \bigcup_{A>0} \mathcal{LJ}_A.
 \end{align*}
 Then the Calabi flow starting from any $J_s \in \mathcal{LJ}'$ converges to $\psi^*(J_0)$,
 in the smooth topology of sections of $TM \otimes T^*M$, where $\psi \in Symp(M,\omega)$ depends on $J_s$.
 \label{thmin:4}
\end{theoremin}

It is interesting to compare the Calabi flow and the K\"ahler Ricci flow on Fano manifolds at the current stage.
For simplicity, we fix $[\omega]=2\pi c_1(M, J)$.
Modulo the pioneering work of H.D. Cao(\cite{[CaoH]}, global existence) and G. Perelman(\cite{[SeT]}, scalar curvature bound),
Theorem~\ref{thmin:3} and  Theorem~\ref{thmin:4} basically says that the convergence of the Calabi flow
can be as good as that  for the K\"ahler Ricci flow on Fano manifolds, whenever some critical metrics are assumed to
exist, in a broader sense.
The K\"ahler Ricci flow version of Theorem~\ref{thmin:3} and Theorem~\ref{thmin:4} has been studied by G. Tian and X.H. Zhu in~\cite{TZ1} and ~\cite{TZ2},
based on Perelman's fundamental estimate.   A more general apporach was developed by G. Sz{\'e}kelyhidi and T.C. Collins in~\cite{[CoSz]}.
Our proof of  Theorem~\ref{thmin:3} and Theorem~\ref{thmin:4} uses a similar strategy to that in Tian-Zhu's work~\cite{TZ2}.
However, our proof is based on backward regularity improvement theorems, Theorem~\ref{thm:HL29_1} and Theorem~\ref{thm:HL29_2},
which are the main technical new ingredients of this paper.

If the flows develop singularity at time infinity, then the behavior of the Calabi flow and the K\"ahler Ricci flow seems much different.
Based on the fundamental work of Perelman, we know collapsing does not happen along the K\"ahler Ricci flow.
In~\cite{CW1} and~\cite{CW2}, it was proved by X.X. Chen and the second author that the K\"ahler Ricci flow will converge to a K\"ahler Ricci soliton flow on a $Q$-Fano variety.
A different approach was proposed in complex dimension 3 in~\cite{TZZ2}, by G.Tian and Z.L. Zhang.
However, under the Calabi flow,  G. Sz{\'e}kelyhidi \cite{MR2571959} has shown that collapsing may happen at
time infinity, by constructing examples of global solutions of the Calabi flow on ruled surfaces.
In this sense, the Calabi flow is much more complicated.
Of course, this is not surprising since we do not specify the underlying K\"ahler class.   A more fair comparison should be between the Calabi flow and the K\"ahler Ricci flow,
in the same class $2\pi c_1(M, J)$, of a given Fano manifold.  However, few is known about the Calabi flow in this respect, except the underlying manifold
is a toric Fano surface (c.f.~\cite{[ChenHe2]}).


Theorem~\ref{thmin:1} and Theorem~\ref{thmin:2} push the difficulty of the Calabi flow study on K\"ahler surfaces
to the proof of global existence, i.e., Chen's conjecture.
Theorem~\ref{thmin:3} and Theorem~\ref{thmin:4} indicate that the study of the Calabi flow with bounded scalar curvature is important.
It is not clear whether the global existence always holds.  If global existence fails, what will happen?
In other words,  what is the best condition for the global existence of the Calabi flow?
Whether the scalar curvature bound is enough to guarantee the global existence?
In order to answer these questions, we can borrow ideas from the study of the Ricci flow.
In~\cite{[Sesum]}, N. Sesum showed that the Ricci flow exists as long
as the Ricci curvature stays bounded.  Same conclusion holds for the Calabi flow, due to the work~\cite{[ChenHe1]} of X.X. Chen and W.Y. He.
However, we can also translate Sesum's result into the Calabi flow along another route.  Note that the Calabi flow satisfies equation (\ref{eq000}).
So the metrics evolve by $\nabla \bar{\nabla} S$, the complex Hessian of the scalar curvature.
Correspondingly, under the Ricci flow, the metrics evolve by $-2R_{ij}$.   Modulo constants, we can regard $\nabla \bar{\nabla} S$ as the counterpart
of Ricci curvature in the Calabi flow.
Consequently,  one can expect that the Calabi flow has global existence whenever $|\nabla \bar{\nabla} S|$ is bounded.
This is exactly the case.
To state our results precisely, we introduce the notations
\begin{align}
O_g(t)=\sup_M\;|S|_{g(t)},\quad P_g(t)=\sup_M \left|\bar{\nabla} \nabla S \right|_{g(t)},
\quad Q_g(t)=\sup_M\;|Rm|_{g(t)}.
\label{eqn:SL06_1}
\end{align}
We shall omit $g$ and $t$ if they are clear in the context.

\begin{theoremin}
Suppose that $\left\{(M^n, g(t)), -T \leq t <0\right\}$ is a Calabi flow solution and $t=0$ is the singular time. Then we have
\begin{align}
  \limsup_{t \to 0} P|t| \geq \delta_0, \label{eqn:HG22_1}
\end{align}
where $\delta_0=\delta_0(n)$.
Furthermore, for each $\alpha \in (0, 1)$, we have
\begin{align}
 \limsup_{t \to 0} O^{\alpha}Q^{2-\alpha}|t| \geq C_0,  \label{eqn:HG22_2}
\end{align}
where $C_0=C_0(n,\alpha)$.
In particular, if $t=0$ is a singular time of type-I, then we have
 \begin{align}
    \limsup_{t \to 0}O^2|t|>0. \label{eqn:HG22_3}
 \end{align}
\label{thmin:5}
\end{theoremin}

Theorem~\ref{thmin:5} is nothing but the Calabi flow counterpart of the main theorems in \cite{[Wang2]} and \cite{[Wang3]}.
The tools we used in the proof of Theorem~\ref{thmin:5}  are motivated by the study of the analogue question of the Ricci flow
by the second author in~\cite{[Wang3]} and~\cite{[Wang2]}.
Actually,  the methods in~\cite{[Wang2]} and~\cite{[Wang3]} were built in a quite general frame.
It was expected to have its advantage in the study of the general geometric flows.\\

The paper is organized as follows.   In section 2, we develop two concepts---curvature scale and harmonic scale--- to study geometric flows.
Based on the analysis of these two scales under the Calabi flow, we show global backward regularity improvement estimates.
In section 3, we combine the regularity improvement estimates, the excellent behavior of the Calabi functional along the Calabi flow
and the deformation techniques to prove Theorem~\ref{thmin:1}-Theorem~\ref{thmin:4}.
In section 4, we first show some examples where Theorem~\ref{thmin:1}-Theorem~\ref{thmin:4} can be applied and then prove Theorem~\ref{thmin:5}. \\

 {\bf Acknowledgements}: The authors would like to thank Professor Xiuxiong Chen, Simon Donaldson, Weiyong He, Claude LeBrun and Song Sun
for insightful discussions.  Haozhao Li and Kai Zheng would like to express their deepest gratitude to Professor Weiyue Ding for his support, guidance and encouragement during the project.
Part of this work was done while Haozhao Li was visiting MIT and he wishes to thank MIT for their generous hospitality.

\section{Regularity scales}
\label{sec:regularity}

\subsection{Preliminaries}\label{Preliminaries}

 Let $M^n$
be a  compact K\"ahler manifold of complex dimension $n$ and $g$  a K\"ahler metric on $M$
with the K\"ahler form $\omega$.  The K\"ahler class corresponding to $\omega$ is denoted by $\Omega=[\omega]$.
The space of K\"ahler potentials is defined by
\begin{align*}
\mathcal{H}(M, \omega)=\left\{\varphi\in C^{\infty}(M)\; \left| \;\omega+\sqrt{-1} \partial \bar{\partial} \varphi>0  \right. \right\}.
\end{align*}
In \cite{[Cal1]} and \cite{[Cal2]}, Calabi introduced the Calabi
functional
$$Ca(\omega_{\varphi})=\int_M\; \left(S(\omega_{\varphi})-\underline{S} \right)^2\,\omega_{\varphi}^n, $$
where $S(\omega_{\varphi})$ denotes the scalar curvature of the metric
 $\omega_{\varphi}=\omega+\sqrt{-1} \partial \bar{\partial} \varphi$ and $\underline{S}$ is the average of the scalar curvature $S(\omega_{\varphi})$.  The gradient
flow of the Calabi functional is called the Calabi flow, which can be written by a parabolic equation of K\"ahler potentials:
\begin{align*}
 \frac{\partial \varphi}{\partial t}=S(\omega_{\varphi})-\underline{S}.
\end{align*}
The metrics evolve by the equation:
\begin{align*}
 \frac{\partial}{\partial t} g_{i\bar{j}}=S_{,i\bar{j}}.
\end{align*}
Consequently, the evolution equations of metric-related quantities are determined by direct calculation:
\begin{align}
  &\frac{\p}{\p t} \Gamma^{k}_{ij}=g_\vphi^{k\bar l}\nabla_j S_{i\bar l},  \label{eqn:J03_1}\\
  &\left( \frac{\p}{\p t}+\Delta^2_\vphi \right) R_{i\bar j k \bar l}
     =-R_{k\bar l\bar r s}(R_{i\bar j})_{r\bar s}-(R_{p\bar{q}i\bar{j}}R_{q\bar{p}}-R_{i\bar{p}}R_{p\bar{j}})_{k\bar l} \notag\\
  &\qquad \qquad \qquad \quad -\left\{R_{\bar m n}R_{i\bar n p \bar l} R_{m\bar j k \bar p}
+ (R_{i\bar n })_p(R_{n\bar j k \bar l})_{\bar p} +(R_{i\bar n k\bar l})_p(R_{n\bar j })_{\bar p}\right\}, \label{eqn:J03_2}\\
 &\left( \frac{\p }{\p t} +\Delta^2_\vphi \right) R_{i\bar j}
 =-R_{i\bar j \bar k l}S_{k\bar l}-\Delta_\vphi (R_{p\bar{q}i\bar{j}}R_{q\bar{p}}-R_{i\bar{p}}R_{p\bar{j}}), \label{eqn:J03_3}\\
 &\left( \frac{\p }{\p t} +\Delta^2_\vphi \right)S =-S^{i\bar j} R_{i\bar j}. \label{eqn:J03_4}
\end{align}
The evolution equation of curvature tensor (\ref{eqn:J03_2}) can be simplified as follows:
\beq \label{eq:Rm}
\frac{\partial}{\partial t} \Rm=-\Na\bar \Na\Na\bar \Na S
=-\Delta^2\Rm+\Na^2\Rm*\Rm+\Na \Rm*\Na \Rm,
\eeq
where the operator $*$ denotes some contractions of tensors.  Thus, we have the inequality
\beq
\pd {}t|\Rm|\leq \left| \Na^4S \right|+c(n)|\Rm| \left|\Na^2S \right|. \label{eq:Rm2}
\eeq

\subsection{Estimates based on curvature bound}

The global high order regularity estimate of the Calabi flow was studied in~\cite{[ChenHe2]}, when Riemannian curvature and Sobolev constant are bounded uniformly.
Taking advantage of the localization technique developed in~\cite{Glick} and~\cite{[Streets]}, one can localize the estimate in~\cite{[ChenHe2]}.

\begin{lemma}
 Suppose $\{(\Omega^n, g(t)), 0 \leq t \leq T\}$ is a Calabi flow solution on an open K\"ahler manifold $\Omega$,  $\overline{B_{g(T)}(x,r)}$ is a geodesic complete ball in $\Omega$.
 Suppose
 \begin{align}
    &C_S(B_{g(T)}(x,r), g(T)) \leq K_1,  \label{eqn:GA04_3}\\
    &\sup_{B_{g(T)}(x,r) \times [0, T]} \left\{ |Rm| + \left| \frac{\partial}{\partial t} g\right| + \left| \nabla \frac{\partial}{\partial t} g\right| \right\} \leq K_2.    \label{eqn:GA04_4}
 \end{align}
 Then for every positive integer $j$, there exists $C=C(j,\frac{1}{r},K_1,K_2)$ such that
 \begin{align*}
    \sup_{B_{g(T)}(x, 0.5 r)} |\nabla^j Rm|(\cdot, T) \leq C.
 \end{align*}
\label{lma:GA04_1}
\end{lemma}

\begin{proof}
 This follows from the same argument as Theorem 4.4 of~\cite{[Streets]} and the Sobolev embedding theorem.
\end{proof}

\begin{lemma}
Let $(M^n, g, J)$ be a complete K\"ahler manifold with $|\Rm|+|\Na \Rm|\leq C_1$. Then there exist
positive constants $r_1, r_2$  depending only on $ C_1, n$ such that for each $p\in M$ there is a map $\Phi$
from the Euclidean ball $\hat B(0, r_1)$ in $\CC^n$ to $M$ satisfying the following properties.
\begin{enumerate}
  \item[(1)] $\Phi$ is a local  biholomorphic map from $\hat B(0, r_1)$ to its image.
  \item[(2)] $\Phi(0)=p$.
  \item[(3)] $\Phi^*(g)(0)=g_{E}$, where $g_E$ is the standard metric on $\CC^n$.
  \item[(4)] $r_2^{-1}g_E\leq \Phi^*g\leq r_2 g_E$ in $\hat B(0, r_1)$.
\end{enumerate}
\label{lem:map}
\end{lemma}

\begin{proof}
This is only an application of Proposition 1.2 of Tian-Yau \cite{[TY1]}. Similar application can be found in~\cite{[ChauTam]}.
\end{proof}

\begin{theorem}[J.Streets~\cite{[Streets]}]
Suppose $\{(M^n, g(t)), 0 \leq t \leq T\}$ is a Calabi flow solution satisfying
\begin{align*}
  \sup_{M \times [0, T]} |\Rm| \leq K.
\end{align*}
Then we have
\begin{align}
&\sup_{x \in M} \left|\Na^l \Rm \right|(x, t)\leq C\left(K+\frac 1{\sqrt{t}}\right)^{1+\frac {l}{2}},  \label{eqn:GA04_1} \\
&\sup_{x \in M} \left|\frac{\partial^l}{\partial t^l} \Rm \right|(x, t) \leq C \left(K+\frac 1{\sqrt{t}}\right)^{1+2l},    \label{eqn:GA04_2}
\end{align}
for every $t \in (0, T]$ and positive integer $l$.  Here $C=C(l,n)$.   In particular, we have
 \begin{align*}
 \sup_{x \in M} \left|\pd {}t|\Rm|\right|(x,T) \leq C \left(K^3+\frac 1{T^{\frac 32}}\right).
 \end{align*}
\label{thm:GA04_JS}
\end{theorem}

\begin{proof}
By equation (\ref{eqn:J03_2}),  we see that (\ref{eqn:GA04_2}) follows from (\ref{eqn:GA04_1}).
We shall only prove (\ref{eqn:GA04_1}).

We argue by contradiction.  As in~\cite{[Streets]}, we define function
\begin{align*}
f_l(x, t, g)=\sum_{j=1}^l \left|\Na^j \Rm \right|_{g(t)}^{\frac 2{2+j}}(x).
\end{align*}
Suppose (\ref{eqn:GA04_1}) does not hold uniformly for some positive integer $l$.
Then there
exists a sequence of Calabi flow solutions $\{(M_i^n, g_i(t)), 0 \leq t \leq T\}$
satisfying the assumptions of the theorem and there are  points $(x_i, t_i)\in M_i\times [0, T]$ such that
$$\lim_{i\ri +\infty}\frac {f_l(x_i, t_i, g_i)}{K+t_i^{-\frac{1}{2}}}=\infty.$$
Suppose that the maximum of $\frac {f_l(x, t, g_i)}{K+t_i^{-\frac 12}}$ on $M\times (0, T]$ is achieved at $(x_i, t_i)$.
We can rescale the metrics by
$$\td g_i(x, t) \triangleq \la_i g \left(x, t_i+\lambda_i^{-2}t \right),\quad \la_i \triangleq f_l(x_i, t_i, g_i).$$
By construction, $t_i\la_i^2\geq 1$ for $i$ large and the flow $\td g_i(t)$ exists on the time period $[-1, 0]$.
Moreover, it satisfies the following properties.
\begin{itemize}
\item $\displaystyle \lim_{i\ri +\infty} \sup_{M_i \times [-1,0]} |\Rm|_{\td g_i}=0$.
\item $\displaystyle \sup_{M_i \times [-1,0]} f_l(\cdot, \cdot, \tilde{g}_i(t)) \leq 1$.
\item $\displaystyle f_l(x_i, 0, \td g_i)=1$.
\end{itemize}
By Lemma \ref{lem:map}, we can construct local biholomorphic map $\Phi_i$ from a ball $\hat B(0, r) \subset \CC^n$ to $M_i$, with respect to the metric $\tilde{g}_i(0)$ and base point $x_i$.
Note that the radius $r$ is independent of $i$.
Let $\tilde{h}_i(t)=\Phi_i^* \tilde{g}_i(0)$. Then we obtain a sequence of Calabi flows $\{(\hat{B}(0,r), \tilde{h}_i(t)), -1 \leq t \leq 0\}$ satisfying (\ref{eqn:GA04_3}) and (\ref{eqn:GA04_4}),
up to shifting of time. Furthermore, we have
\begin{align}
  \lim_{i \to \infty} |Rm|_{\tilde{h}_i(0)}(0)=0, \quad \sum_{j=1}^l \left|\Na^j \Rm \right|_{\tilde{h}_i(0)}^{\frac 2{2+j}}(0)=1.
\label{eqn:GA04_5}
\end{align}
By Lemma~\ref{lma:GA04_1}, we can take convergence in the smooth Cheeger-Gromov topology.
\begin{align*}
  \left( \hat{B}(0, 0.5 r), \tilde{h}_i(0) \right)   \longright{Cheeger-Gromov-C^{\infty}}  \left(\tilde{B},  \tilde{h}_{\infty}(0) \right).
\end{align*}
On one hand, $\Rm_{\td h_{\infty}(0)} \equiv 0$ on $\tilde{B}$, which in turn implies that $\nabla^j Rm_{\td h_{\infty}(0)} \equiv 0$ on $\tilde{B}$ for each positive integer $j$.
On the other hand,  taking smooth limit of (\ref{eqn:GA04_5}), we obtain
\begin{align*}
  \sum_{j=1}^l \left|\Na^j \Rm \right|_{\tilde{h}_{\infty}(0)}^{\frac 2{2+j}}(0)=1.
\end{align*}
Contradiction.
\end{proof}

The proof of Theorem~\ref{thm:GA04_JS} follows the same line as that in~\cite{[Streets]} by J. Streets, we do not claim the originality of the result.
We include the proof here for the convenience of the readers and to show the application of the local biholomorphic map $\Phi$, which will be 
repeated used in the remainder part of this subsection. 
Actually, by delicately using interpolation inequalities,  the constants in Theorem~\ref{thm:GA04_JS} can be made explicit.   \\


Note that for a given Calabi flow,  $S \equiv 0$ implies all the high derivatives of $S$ vanish.
Therefore, for the Calabi flows with uniformly bounded Riemannian curvature and very small scalar curvature $S$, it is expected that the high derivatives of $S$ are very small.
In the remainder part of this subsection, we will justify this observation.   Similar estimates for the Ricci flow were given by Theorem 3.2 of \cite{[Wang2]} and Lemma
2.1 of \cite{[Wang3]} by the parabolic Moser iteration. However,  since the Calabi flow is a fourth-order
parabolic equation, the parabolic Moser iteration in the case
of the Ricci flow does not  work any more.
Here we use a different method to overcome this difficulty.

To estimate the higher derivatives of the curvature, we need the
interpolation inequalities of Hamilton in \cite{[Ha82]}:
\begin{lemma}\label{lem:Ha82} (\cite{[Ha82]}) For any tensor $T$ and $1\leq j\leq
k-1$, we have
\begin{align}
&\int_M\;|\Na^j T|^{2k/j}\,\omega^n\leq C \cdot
\max_M|T|^{\frac {2k}{j}-2} \cdot \int_M\;|\Na^k T|^2\,\omega^n, \label{eqn:GA02_1}\\
 &\int_M\;|\Na^j T|^{2 }\,\omega^n
\leq C \cdot \left(\int_M\;|\Na^k T|^2\,\omega^n\right)^{\frac jk} \cdot
 \left(\int_M\;|  T|^2\,\omega^n\right)^{1-\frac jk}  \label{eqn:GA02_2}
\end{align}
where $C=C(k, n)$ is a constant.

\end{lemma}

Combining Lemma \ref{lem:Ha82} with Sobolev embedding theorem, we have the following result.

\begin{lemma}\label{lem:tensor}For any integer $i\geq 1$ and any
K\"ahler metric $\oo$,  there exists a constant $C=C(C_S(\oo), i)>0$
such that
 for any tensor $T$, we have
\beqs \max_M \left| \Na^iT \right|^2&\leq& C\cdot\max_M\;|T|^{\frac
{2n+1}{n+1}}\left(\int_M\;|T|^2\,\oo^n\right)^{\frac1{4(n+1)}}\\
&&\cdot
\left(\int_M\;\left|\Na^{4(n+1)i}T\right|^2\,\oo^n+\int_M\; \left| \Na^{4(n+1)(i+1)}T \right|^2\,\oo^n\right)^{\frac
1{4(n+1)}}.\eeqs

\end{lemma}

\begin{proof}
Recall that the Sobolev embedding theorem implies that
 \begin{align*}
 \max_M|f|\leq  C_S\left(\int_M\;(|f|^p+|\Na f|^p)\,\oo^n\right)^{\frac1p}
 \end{align*}
 for every smooth function $f$.
 Let $f=\left|\Na^iT \right|^2$ and $p=2(n+1)$ in the above inequality.  Then we have
\begin{align*}
 \max_M \left| \Na^iT \right|^2
 \leq
C_S\left(\int_M\left( \left| \Na^iT \right|^{4(n+1)}+|\Na f|^{2(n+1)} \right)\,\oo^n\right)^{\frac 1{2(n+1)}}.
\end{align*}
The Kato's inequality implies that
\beqs
 |\Na f|=2 \left| \Na^i T \right|
 \cdot  \left|\Na \left|\Na^i T\right| \right| \leq 2 \left|\Na^i T \right| \cdot \left| \Na^{i+1} T \right| \leq  \left| \Na^i T \right|^2
 +  \left| \Na^{i+1} T \right|^2.
\eeqs
Combining the above two inequalities, we obtain
\begin{align}
\label{eq012}
\max_M \left| \Na^iT \right|^2
\leq C\left(\int_M\; \left( \left|\Na^iT \right|^{4(n+1)}+ \left| \Na^{i+1} T \right|^{4(n+1)} \right)\,\oo^n\right)^{\frac 1{2(n+1)}}.
\end{align}
In inequality (\ref{eqn:GA02_1}), let $k=2(n+1)i$ and $j=i$, we have
 \beqs \int_M\; \left| \Na^i T \right|^{4(n+1)}\,\omega^n&\leq&C\cdot
\max_M |T|^{4n+2} \cdot \int_M\; \left| \Na^{2(n+1)i} T \right|^2\,\omega^n. \\
 \nonumber
 \eeqs
If we let $k=2(n+1)(i+1)$, $j=i+1$, then we have
\beqs \int_M\; \left|\Na^{i+1} T \right|^{4(n+1)}\,\omega^n
&\leq&
C\cdot \max_M |T|^{4n+2}  \cdot \int_M\; \left| \Na^{2(n+1)(i+1)} T \right|^2\,\omega^n. \\
 \nonumber
\eeqs
In inequality (\ref{eqn:GA02_2}),  let $k=4(n+1)i$ and $j=2(n+1)i$, then we have
\beqs
\int_M\; \left|\Na^{2(n+1)i}T \right|^{2}\,\oo^n&\leq&
C\cdot \left(\int_M\; \left| \Na^{4(n+1)i}T \right|^2\right)^{\frac 12}
\cdot
\left(\int_M\;|T|^2\,\oo^n \right)^{\frac12}.
\nonumber
\eeqs
 Combining this with
(\ref{eq012}), we have
 \beqs &&\max_M \left| \Na^iT \right|^2\\
&\leq&
C\cdot\max_M\;|T|^{\frac{2n+1}{n+1}}
\left(\int_M\; \left(\left|\Na^{2(n+1)i}T \right|^{2}+ \left|\Na^{2(n+1)(1+i)}T \right|^{2} \right)\,\oo^n\right)^{\frac 1{2(n+1)}}\\
&\leq&C\cdot\max_M\;|T|^{\frac{2n+1}{n+1}}
\left(\int_M\;|T|^2\,\oo^n\right)^{\frac1{4(n+1)}}
\left(\int_M\; \left|\Na^{4(n+1)i}T \right|^2\,\oo^n +\int_M\; \left|\Na^{4(n+1)(i+1)}T \right|^2\,\oo^n\right)^{\frac
1{4(n+1)}}.\eeqs The lemma is proved.
\end{proof}

We next show some local estimates on the derivatives of the scalar
curvature.
\begin{lemma}\label{lem:tensor2}
Fix any $\al\in (0, 1)$ and $r>0$.  There exists an integer $N=N(\al)>0$
such that if
\beq
\sup_{B(p, 2 r)}\sum_{k=0}^N\, \left| \Na^k \Rm \right| \leq \La_N
\label{eq607}
\eeq
for some positive constant $\La_N$, then we have the inequalities
\beqn
\sup_{B(p, r)} \left|\Na \bar \Na S \right|
&\leq& C\sup_{B(p,2r)}|S|^{\al}, \label{eq608}\\
\sup_{B(p, r)} \left|\Na\bar \Na\Na \bar \Na S \right| &\leq& C\sup_{B(p, 2r)} \left| \Na\bar \Na S \right|^{\al},
\label{eq608a}\eeqn for some constant $C=C\left(C_S(B(p, 2r)), r, \La_N, \al \right)$.

\end{lemma}

\begin{proof}Let $\eta\in C^{\infty}(\RR, \RR)$ be a cutoff function
such that $\eta(s)=1$ if $s\leq 1$ and $\eta(s)=0$ if $s\geq 2.$
Moreover, we assume $\left|\eta'(s) \right|\leq 2. $ We define
$\chi(x)=\eta\left(\frac {d_g(p,x)}{r}\right)$ for any $x\in M$, where $d_g(p, x)$ is the distance
from $p$ to $x$ with respect to the metric $g$. Then
 $\chi=1$ on $B(p, r)$ and $\chi=0$
outside $B(p, 2r).$ The derivatives of $\chi$ satisfy the inequalities
\beq \left|\Na_g^i\chi \right| \leq C(\La_i, r).\label{eq610} \eeq
Using Lemma \ref{lem:tensor} for $\chi S$,  we have the inequality
 \begin{align*}
  \sup_{B(p, r)}\;\left|\Na^2 S \right|^2
  &\leq C(C_S)\cdot\sup_{B(p, 2r)}\;|S|^{\frac {2n+1}{n+1}}
\left(\int_{B(p,2r)}\;|S|^2\,\oo^n\right)^{\frac1{4(n+1)}}\\
  &\quad \cdot \left(\int_{B(p, 2r)}\;\left|\Na^{8(n+1)}(\chi S)\right|^2\,\oo^n+
\int_{B(p, 2r)}\; \left|\Na^{12(n+1)}(\chi S) \right|^2\,\oo^n\right)^{\frac1{4(n+1)}}.
\end{align*}
 For any integer $k$ we set
 $$f(k)=\int_{B(p, 2r)}\; \left| \Na^{k}(\chi S) \right|^2\,\oo^n.$$
Under the assumption (\ref{eq607}), we have
\begin{align*}
f(k)&\leq  \int_{B(p, 2r)}\; |S| \left|\Na^{2k}(\chi S) \right|\,\oo^n\\
&\leq \vol(B(p, 2r))^{\frac 12} \sup_{B(p, 2r)}\; |S| \cdot f(2k)^{\frac 12}\\
&\leq \Big(\vol(B(p, 2r))^{\frac 12}
\sup_{B(p, 2r)}\; |S|\Big)^{2-2^{1-m}}f(2^mk)^{\frac 1{2^m}}\\
&\leq C(\La_{2^mk}, m, r  )\Big(\vol(B(p, 2r))^{\frac 12} \sup_{B(p, 2r)}\; |S|\Big)^{2-2^{1-m}}
\end{align*}
where $m$ is any positive integer. Here we used the fact that
$\vol(B(p, 2r)$ is bounded from above by the volume comparison
theorem. Combining the above
inequalities, for any $m$ we have the estimate
$$\sup_{B(p, r)}\; \left|\Na^2 S \right| \leq C \left(C_S, r, m, \La_{12\cdot 2^m(n+1)} \right)\sup_{B(p, 2r)}\; |S|^{1-\frac 1{(n+1)2^{m+2}}}.$$
Thus, (\ref{eq608}) is proved.
Applying Lemma \ref{lem:tensor} to $\Na \bar \Na S$ and using the
same argument as above, we have the   inequality (\ref{eq608a}).  The lemma is proved.

\end{proof}

The following proposition is a weak version of the corresponding result for the Ricci flow in \cite{[Wang2]}\cite{[Wang3]}.

\begin{proposition}\label{lem:003}Fix $\al\in (0, 1)$.
If $\{(M, g(t)), -s \leq t \leq 0\}$ is a Calabi flow solution with $\displaystyle \sup_{[-s, 0]}Q_g(t)\leq 1$, then there is a
constant $C=C(s, \al)>0$  such that
\beqn
\max_M\, \left|\Na\bar \Na S \right|(0)&\leq& C
\max_M\,|S|^{\al}(0), \label{eq300}\\
\max_M\, \left|\Na\bar \Na\Na\bar \Na S \right|(0)&\leq& C
\max_M\, \left|\Na\bar \Na S \right|^{\al}(0). \label{eq300a}
\eeqn
\end{proposition}
\begin{proof}
Since $\displaystyle \sup_{[-s, 0]}Q_g(t)\leq 1$, Theorem~\ref{thm:GA04_JS} applies.
For any integer $k\geq 0$, there is a constant $C=C(k, s)$ such that \beq
\sup_{M\times[-\frac {s}2, 0]}|\Na^k \Rm|\leq C(k, s).
\label{eq605}\eeq By Lemma \ref{lem:map}, for any $p\in M$ there is $r=r(n, s)>0$ and a local biholomorphic map $\Phi$ from $ \hat B(0, r)\subset \CC^n$ to its
image in $(M, g(0))$ and $\Phi(0)=p$. Define the pullback metrics $ \hat
g(t)=\Phi^*g(t) $ on $\hat B(0, r)$.  Then the K\"ahler metric $\hat g(t)$  satisfies
the equation of Calabi flow on $\hat B(0, r)$ and their injectivity radii on $\hat B(0, r)$ are
bounded from below.  Moreover, on $\hat B(0, r)$ the Sobolev constant and all the derivatives of the curvature
tensor of $\hat g(t)$ for $t\in [-\frac {s}2, 0]$  are bounded by
(\ref{eq605}).  By Lemma \ref{lem:tensor2}, there is a constant
$C(s, \al)>0$ such that
$$\sup_{\hat B(0, \frac r2)} \left| \Na\bar \Na S \right|_{\hat g}(0)\leq C\max_{\hat B(0, r)}|S|_{\hat g}^{\al}(0),$$
which implies that the  metric $g(t)$ satisfies the inequality
(\ref{eq300}).
Similarly, we can show the second inequality (\ref{eq300a}).
The proposition is proved.
\end{proof}

\subsection{From metric equivalence to curvature bound}

If we regard curvature as the 4-th order derivative of K\"ahler potential function,
then Theorem~\ref{thm:GA04_JS} can be roughly understood as from $C^4$-estimate to $C^l$-estimate, for each $l \geq 5$.
In this subsection, we shall setup the estimate from $C^2$ to $C^4$ for the Calabi flow family.

\begin{lemma}\label{lem001}
For every $\dd>0$, there exists a constant $\ee=\ee(n, \dd)$ with
the following properties. If $\{(M^n, g(t)), t \in [-1,0]\}$ and $\{(M^n, h(t)), t\in [-1, 0]\}$
are  Calabi flow solutions such that
\begin{itemize}
  \item $Q_g(0)=1$ and $Q_g(t)\leq 2$ for any $t\in [-1,
  0]$.
  \item $Q_h(t)\leq 2$ for any $t\in [-1, 0]$.
  \item $e^{-\ee} g(0)\leq h(0)\leq e^{\ee}g(0)$.
\end{itemize}
Then we have
$$|\log Q_h(0)|<\dd.$$
\end{lemma}
\begin{proof}
Suppose not, there exist constants $\dd_0>0, \ee_i\ri 0$ and two
sequences of the Calabi flow  solutions
\begin{align*}
 \left\{ (M_i, g_i(t)), t\in [-1, 0] \right\},  \quad \left\{ (M_i, h_i(t)), t\in [-1, 0] \right\}
\end{align*}
such that the following properties are satisfied.
\begin{enumerate}
  \item[(1)] $Q_{g_i}(0)=1$ and $Q_{g_i}(t)\leq 2$ for any $t\in [-1, 0]$.
  \item[(2)] $Q_{h_i}(t)\leq 2$ for any $t\in [-1, 0]$, and $|\log Q_{h_i}(0)|\geq \dd_0$.
  \item[(3)] $e^{-\ee_i} g_i(0)\leq h_i(0)\leq e^{\ee_i}g_i(0)$.
\end{enumerate}
By property (1) and (2),  we claim that there is a point $z_i\in M_i$ such that
\beq
\max\{|\Rm|_{h_i(0)}(z_i), |\Rm|_{g_i(0)}(z_i)\}\geq e^{-2\dd_0},\quad
\left|\log \frac {|\Rm|_{h_i(0)}(z_i)}{|\Rm|_{g_i(0)}(z_i)}\right|\geq
\frac 12 \dd_0.\label{eq500}
\eeq
In fact, by property (2) we have two possibilities $Q_{h_i}(0)\geq e^{\dd_0}$ or $Q_{h_i}(0)\leq e^{-\dd_0}.$ If $Q_{h_i}(0)\geq e^{\dd_0}$, then we assume that the point $z_i$ achieves  $Q_{h_i}(0)$:
$$|\Rm|_{h_i(0)}(z_i )= Q_{h_i}(0). $$
Then the first inequality of (\ref{eq500}) obviously holds and  the second also holds since
$$\log \frac {|\Rm|_{h_i(0)}(z_i)}{|\Rm|_{g_i(0)}(z_i)}\geq \log \frac {|\Rm|_{h_i(0)}(z_i)}{Q_{g_i}(0)}\geq \dd_0.$$
Now we consider the case when $Q_{h_i}(0)\leq e^{-\dd_0}.$ We assume that the point $z_i$ achieves  $Q_{g_i}(0)=1$:
$$|\Rm|_{g_i(0)}(z_i )= 1. $$
Then the first inequality of (\ref{eq500}) obviously holds and  the second also holds since
$$\log \frac {|\Rm|_{g_i(0)}(z_i)}{|\Rm|_{h_i(0)}(z_i)}\geq \log e^{\dd_0}=\dd_0.$$
Thus, (\ref{eq500}) is proved.

By Theorem~\ref{thm:GA04_JS}, we have the higher order curvature
estimates for the metrics $g_i(0)$.
By Lemma \ref{lem:map}, there is $r>0$ independent of $i$ and a local biholomorphic map
$\Phi_i$ from $\hat B(0, r)\subset \CC^n$ to its image in $M_i$ and $\Phi_i(0)=z_i$.
Define the pullback metrics
\beq
\hat g_i(t)=\Phi_i^*g_i(t),\quad \hat
h_i(t)=\Phi_i^*h_i(t).\label{eq501}
\eeq
Then the K\"ahler metrics $\hat g_i(t)$ and $\hat
h_i(t)$ satisfy the equation of the Calabi flow and their injectivity
radii at the point $0\in \hat B(0, r) $ are bounded from below by a uniform positive constant which is
independent of $i$. Moreover, all the derivatives of the curvature
tensor of $\hat g_i(t)$ and $\hat
h_i(t)$ are bounded by Theorem~\ref{thm:GA04_JS}. Thus, we can take Cheeger-Gromov smooth convergence
\beq
\left( \hat B(0, r ), 0, \hat g_i(t) \right) \longright{C.G.-C^{\infty}} \left( B', p', \hat g' \right),\quad
\left( \hat B(0, r ), 0, \hat h_i(t) \right)\longright{C.G.-C^{\infty}}  \left(B'', p'', \hat h' \right).
\label{eq502}
\eeq
Define the identity map $F_i$ by
\beqs
F_i:  \left( \hat B(0, r ), 0, \hat g_i(t) \right) &\mapsto& \left( \hat B(0, r ), 0, \hat h_i(t) \right),\\
x&\mapsto &x.
\eeqs
By property (3), we see that
\begin{align*}
\left|\log \frac {d_{\hat h_i(0)}(x, y)}{d_{\hat g_i(0)}(F_i^{-1}(x), F_i^{-1}(y))}\right|=\left|\log \frac {d_{\hat h_i(0)}(x, y)}{d_{\hat g_i(0)}( x, y)}\right|
\leq \ee_i\ri 0,\quad \forall\;x,y\in \hat B(0,r).
\end{align*}
Therefore, $F_i$ converges to an isometry $F_{\infty}:$
$$F_{\infty}: \left(B', p', \hat g' \right) \ri  \left( B'', p'', \hat h' \right),\quad F_{\infty}(p')=p''. $$
By Calabi-Hartman's theorem in \cite{[CH]}, the isometry $F_{\infty}$ is smooth.
Therefore, we have
\beq |\Rm|_{\hat g'}(p')=|\Rm|_{\hat h'}(p''). \label{eq503}\eeq
However, by the smooth convergence (\ref{eq502}) and the
inequalities (\ref{eq500}) we obtain
\begin{align*}
\max \left\{|\Rm|_{\hat g'}(p'), |\Rm|_{\hat h'}(p'') \right\} \geq e^{-2\dd_0},\quad
\left|\log\frac {|\Rm|_{\hat g'}(p')}{|\Rm|_{\hat h'}(p'')}\right| \geq \frac 12\dd_0,
\end{align*}
which contradicts (\ref{eq503}). The lemma is proved.
\end{proof}

As a direct corollary, we have the next result.

\begin{lemma}\label{lem002}There exists a constant $\ee_0=\ee_0(n)$ with the
following properties.

 Suppose $\{(M^n, g(t)), -1 \leq t \leq K, 0 \leq K\}$ is a Calabi flow solution such that
$$Q_g(0)=1, \quad Q_g(t)\leq 2,\quad \forall\; t\in [-1, 0],$$
 and $T$ is the first time such that $\left|\log Q_g(T) \right|=\log 2$, then we have
\beq \int_0^T\;P_g(t)\,dt\geq \ee_0. \label{eq400}\eeq

\end{lemma}

\begin{proof}
Consider two solutions to the Calabi flow
$\{(M, g(t)), -1 \leq t \leq 0\}$ and $\{(M, h(t)), -1 \leq t \leq 0\}$
where $h(x, t)=g(x, t+T).$ Note that
$|\log Q_h(0)|=\log 2$. By Lemma \ref{lem001} there exists a point $x\in
M$ and a nonzero vector $V\in T_xM$ such that
\begin{align*}
\left|\log \frac {h(0)(V,V)}{g(0)(V, V)}\right| \geq \ee_0
\end{align*}
for some $\ee_0>0.$ Using the equation (\ref{eq:Rm2}), we have
\begin{align*}
\ee_0\leq \left|\log \frac {g(T)(V,V)}{g(0)(V, V)}\right| \leq \int_0^T\,P_g(t)\,dt.
\end{align*}
Thus, (\ref{eq400}) holds and the lemma is
proved.
\end{proof}


\subsection{Curvature scale}

Inspired by Theorem~\ref{thm:GA04_JS}, we define a concept ``curvature scale", to study improving regularity property of the Calabi flow.

\begin{definition}\label{RS}
 Suppose $\{(M,g(t)), t \in I \subset \RR\}$ is a Calabi flow solution.
 Define the curvature scale $F_g(t_0)$ of $t_0 \in I$ by
$$F_g(t_0)=\sup\Big\{s>0\;\Big|\; \sup_{M \times [t_0-s, t_0]} |\Rm|^2 \leq s^{-1}\Big\},$$
where we assume
$
  \displaystyle  \sup_{M \times [t_0-s, t_0]} |\Rm|^2 = \infty
$
 whenever $t_0-s \notin I$.   We denote the curvature scale of time  $t_0$
  by $F_g(t_0)$.
 \end{definition}

Suppose that $\{(M, \td g(t)), -2 \leq t < T\}$ is a Calabi flow solution and $T>0$ is the singular time.
Since the curvature tensor will blow up at the singular time,
we have $\displaystyle \lim_{t\ri T}F_{\td g}(t)=0$. Note that $F_{\td g}(t)$ is a continuous function, we can
assume that $0<F_{\td g}(t)<1$ for any $t\in [0, T)$. Choose any
 $t_0\in [0, T)$ and let $A$ be a positive constant   such
that $A^2=\frac 1{ F_{\td g}(t_0)}>1$.   Rescale the
metric $\td g(t)$ by \beq   g(x, t)=A\,\td g \left(x, \frac {t}{A^2}+t_0 \right),\quad t\in
\left[A^2(-2-t_0), A^2(T-t_0) \right].\label{eq:nor}\eeq Then we obtain a solution $g(t)$ with existence time period containing $[-2, K]$ for some positive $K$.
Moreover, we have $F_g(0)=1$. In the following, we will calculate
the derivative of $F_g(t)$ with respect to $t$. Since $F_g(t)$ might not
be differentiable, we will use the Dini derivative:
\beq \frac{d^-}{dt}f(t):=\liminf_{\ee\ri 0^+}\frac
{f(t+\ee)-f(t)}{\ee},\quad \frac{d^+}{dt}f(t):=\limsup_{\ee\ri 0^+}\frac
{f(t+\ee)-f(t)}{\ee}.
\label{eq:Dini}\eeq

The following result is the key estimate on the curvature scale.
\begin{lemma}\label{lem:800}
Suppose $\{(M,g(t)), -2 \leq t \leq K,  0 \leq  K\}$ is a Calabi flow solution, $F_g(0)=1$.
Then  at time $t=0$  we have \beq
 \frac {d^-}{dt}F_g(t)\geq \min\left\{0,\, -2 \frac {d^+}{dt}Q_g(t)\right\}.
\label{eq705} \eeq
\end{lemma}
\begin{proof}
  By the definition of curvature scale,
 there exists  time $t \in [-1, 0]$ such that $Q_g(t)=1$ and we
 denote by
 $t_0$ the maximal time $t\in [-1, 0]$ with this property.
 There are several cases to consider:

(1). Suppose  $Q_g(0)<1$ and $t_0>-1$. Then there is a constant $\dd\in (0, t_0+1)$
such that
 $Q_g(t)<1$ for all $t\in (0, \dd).$  Thus, for any $t\in [0, \dd)$ we have $F_g(t)=1$. So we obtain
the derivative of $F_g(t)$ at $t=0$,
    \beq
      \frac{d^-}{dt}F_g(t) =0.  \label{eq802a} \eeq

(2). Suppose $Q_g(0)<1$ and $t_0=-1$. By the definition of $t_0$, there exists a small constant
$\dd>0$ such that for
any $t\in (-1, \dd)$ we have $Q_g(t)<1.$ This implies that
$F_g(t)>1$ for $t\in (0, \dd)$ when $\dd$ small. In other words, we
have
\beq
\left. \frac {d^-}{dt}F_g(t)\right|_{t=0}\geq 0.
\eeq

(3). Suppose $Q_g(0)=1$ and there is a small constant $\dd>0$ and a sequence of times
$t_i\in (0, \dd)$ with $t_i\ri 0$ such that
\beq Q_g(t_i) \geq 1,\quad Q_g(t_i)^2F_g(t_i)=1,\quad \forall\;i\geq
1. \label{eq:106}\eeq
Therefore, we have the equality
\beqs \liminf_{i\ri +\infty}\frac{F_g(t_i)-F_g(0)}{t_i}&=& \liminf_{i\ri +\infty}
\frac {Q_g(0)+Q_g(t_i)}{Q_g(0)^2Q_g(t_i)^2}\cdot \frac {Q_g(0)-Q_g(t_i)}{t_i}\\
&=&-\frac 2{Q_g(0)^3}\limsup_{i\ri +\infty}\frac {Q_g(t_i)-Q_g(0)}{t_i}\\&\geq& \left. -\frac 2{Q_g(0)^3}\frac{d^+}{dt}Q_g\right|_{t=0}.
\eeqs

(4). Suppose $Q_g(0)=1$ and there is a small constant $\dd>0$ and a sequence of times
$t_i\in (0, \dd)$ with $t_i\ri 0$ such that
\beq Q_g(t_i) \geq 1,\quad Q_g(t_i)^2F_g(t_i)<1,\quad \forall\;i\geq
1. \label{eq:102}\eeq
Since $F_g(t)$ is a continuous function and $F_g(0)=1$, we assume that
\beq |F_g(t)-1|<\ee_0,\quad \forall\,t\in (0, \dd) \label{eq:103}\eeq
for some small $\ee_0>0.$ For each time $t_i$, (\ref{eq:102}) implies that
\beq Q_g(t_i)^2<\frac 1{F_g(t_i)}=\sup_{[t_i-F_g(t_i),
t_i]}Q_g(t)^2. \label{eq:104}\eeq
Since the constant $\dd$ is small and the inequality (\ref{eq:103}) holds, we
have
$t_i-F_g(t_i)<0<t_i$. Using the fact that $Q_g(0)=1$ and (\ref{eq:104}) we have  $\frac 1{F_g(t_i)}\geq
1$ and so $F_g(t_i)\leq 1.$

If there exists time $t_{i_0}$ such that $F_g(t_{i_0})=1$, then for
all $t\in (0, t_{i_0})$ we have $F_g(t)=1$. Thus, the inequality
(\ref{eq705}) holds. On the other hand, if for all integer $i\geq 1$
the inequality $F_g(t_i)<1$ holds, then we can find a maximal time
$t_i'\in (0, t_i)$ such that
\beq
Q_g(t_i')=\sup_{[t_i-F_g(t_i),
t_i]}Q_g(t)^2. \label{eq:105}
\eeq Moreover, since $Q(t)\leq 1$ for $t\in [-1, 0]$ and
$[t_i'-F(t_i), t_i']\subset [-1, t_i]$, we have
$$Q_g(t_i')=\sup_{[t_i'-F_g(t_i),
t_i']}Q_g(t)^2=\frac 1{F_g(t_i)},$$
where we used (\ref{eq:104}) and (\ref{eq:105}) in the last
equality. Therefore, at time $t_i'$ we have $F_g(t_i')=F_g(t_i)$ and the identity
$$Q_g(t_i')F_g(t_i')=Q_g(t_i')F_g(t_i)=1.$$
Therefore, there exists a sequence of times $\{t_i'\}(t_i'\in (0, t_i))$ with $t_i'\ri 0 $
 and
 satisfying the
conditions (\ref{eq:106}) in item (3). Note that
$$\frac {F(t_i)-F(0)}{t_i}=\frac {F(t_i')-F(0)}{t_i'}\cdot \frac {t_i'}{t_i}\geq
\frac {F(t_i')-F(0)}{t_i'},$$ where we used the fact that
$F(t_i')<F(0)=1$ and $t_i'<t_i$. Therefore, we have
$$\liminf_{i\ri +\infty}\frac{F_g(t_i)-F_g(0)}{t_i}\geq \liminf_{i\ri +\infty}\frac{F_g(t_i')-F_g(0)}{t_i'}
\geq \left. -\frac 2{Q_g(0)^3}\frac{d^+}{dt}Q_g \right|_{t=0}. $$

(5). Suppose $Q_g(0)=1$ and there is a small constant $\dd>0$ and a sequence of times
$t_i\in (0, \dd)$ with $t_i\ri 0$ such that $Q_g(t_i)<1$.
 Then by the definition of $F_g(t)$, we have $F_g(t_i)\geq 1$. Thus,
we have the inequality
$$\liminf_{i\ri +\infty}\frac{F_g(t_i)-F_g(0)}{t_i}\geq 0. $$

Combining the above cases, (\ref{eq705}) holds and the theorem is
proved.

\end{proof}

\begin{lemma}\label{lem:801}

Suppose $\{(M,  g(t)), -2 \leq t \leq K, 0 \leq K\}$ is a Calabi flow solution with $F_g(0)=1$.    Then at time $t=0$ we have
  \begin{align*}
  \left|\frac{dQ_g}{dt}\right| \leq C \min \left\{1, P_g^{\al}, O_g^{\al}
  \right\},
    \end{align*}
where $\alpha$ be any number in $(0,1)$ and $C=C(\al, n)$.
\end{lemma}

\begin{proof} Suppose that $F(0)=1$. By the definition of regularity
scale and Theorem~\ref{thm:GA04_JS}, all higher order derivatives of
the Riemannian curvature tensor are bounded for $t\in [-\frac 12, 0]$.
By Proposition \ref{lem:003}, for any $\al\in (0, 1)$ we have the
estimates
  \beqs
    \max_{M} |\nabla \bar{\nabla}\nabla \bar{\nabla} S| &\leq& C\max_{M} |\nabla \bar{\nabla} S|^{\al} ,\\
       \max_{M} |\nabla \bar{\nabla} S| &\leq & C\max_{M} |S|^{\al}.
\eeqs
Therefore, by the inequality  (\ref{eq:Rm2}) of $|\Rm|$, adjusting $\alpha$ in different inequalities if necessary, we have
    \begin{align}
       \left|\frac{d}{dt}|\Rm| \right| \leq C \min\left\{ 1, \max_{M} |\nabla \bar{\nabla} S|^{\al}, \max_{M}
       |S|^{\al}
       \right\}. \no
    \end{align} The lemma is proved.

    \end{proof}

A direct corollary of the above results is
    \begin{lemma}\label{cor:802} Let $\al\in (0, 1)$. Suppose $\left\{(M, \td g(t)), -2 \leq t < T,  0 \leq T \right\}$ is
 a Calabi flow solution.  For any $t_0\in [-1, T)$ with $F_{\td g}(t_0)\in (0, 1)$,
we have the following inequality at time $t_0$
    \begin{align}
     \frac{d^-}{dt}F_{\td g} \geq   -C\min \left\{1, P_{\td  g}^{\al}Q_{\td g}^{2(1-\al)}F_{\td g},
     O_{\td  g}^{\al}Q_{\td  g}^{2-\al}F_{\td g} \right\},
    \label{eqn:CSH_3}
    \end{align} where $C=C(\al, n)>0$ is  a constant.
    \end{lemma}
\begin{proof}We rescale the metric $\td g(t)$ by (\ref{eq:nor}) with
$A^2=\frac 1{F_{\td g}(t_0)}.$ Then we obtain a solution $g(s)$ to
the
Calabi flow with $F_g(0)=1.$ By Lemma \ref{lem:800} and Lemma
\ref{lem:801}, at $s=0$ we have the following inequality for $g(s)$:
\beq
 \frac{d^-}{dt}F_{g}
\geq -C
\min \left\{1, P_{ g}^{\al}Q_g^{2(1-\al)}F_{g},  O_g^{\al}Q_{ g}^{2-\al}F_{g} \right\}. \label{eq706}
\eeq
Clearly, (\ref{eqn:CSH_3}) follows directly from (\ref{eq706}) and rescaling, since both sides of (\ref{eq706})  are scaling invariant.
\end{proof}

\begin{lemma}\label{cly:SH14_1} Fix $\al\in (0, 1)$.
  Suppose $\left\{ (M^n,g(t)), -1 \leq t \leq K, 0 \leq K \right\}$ is a Calabi flow solution such that
   \begin{enumerate}
     \item[(1)] $Q(0)=1$, where $Q=Q_g$;
     \item[(2)] $Q(t) \leq 2$ for all $t \in [-1,0]$;
     \item[(3)] $|\log Q(t)| \leq \log 2$ for every $t \in [0,K]$ and $|\log Q(K)|=\log 2$.
   \end{enumerate} Then there
 are  constants $A=A(n, \al)$  and $\ee(n, \al)$  such that
 \begin{align}
  &\int_{0}^{K} O^{\alpha}Q^{2-\alpha} dt \geq \frac{\epsilon}{A},  \label{eqn:SH14_4}\\
  &\left|\frac{d}{dt} Q(t) \right| \leq AO^{\alpha} Q^{3-\alpha}, \quad \forall \; t \in [0,K]. \label{eqn:SH15_2}
\end{align}

 \end{lemma}
\begin{proof}Under the   assumptions, we
have the inequalities for $t\in [0, K]$
 \begin{align*}
   P(t) \leq C O^{\alpha}(t) Q^{2-\alpha}(t),\quad
   \max_M \left| \Na^4S \right| \leq C O^{\alpha}(t)Q^{3-\alpha}(t).
 \end{align*} which are the scaling invariant version of Proposition
 \ref{lem:003}. Combining this with Lemma \ref{lem002}, we have
 the   inequalities (\ref{eqn:SH14_4}) and (\ref{eqn:SH15_2}). The
 lemma is proved.

\end{proof}

The next result shows that under the scalar curvature conditions, if the curvature at some time is large enough, then the curvature before that  can be controlled.
 \begin{proposition}\label{theo:reg1}
If $\left\{ (M^n, g(t)), -1 \leq t \leq 0 \right\}$ is a Calabi flow solution with
\begin{enumerate}
  \item[(1)] the scalar curvature $|S(t)|\leq 1$ for every $t\in [-1, 0]$;
  \item[(2)] $Q(0)$ is big enough, i.e.,
   \begin{align*}
     Q(0)> \max \left\{2^{\frac{3-\alpha}{\alpha}} A^{\frac{1}{\alpha}},  2^{\frac{2(1-\alpha)}{\alpha}} \left( \frac{A}{\epsilon}\right)^{\frac{1}{\alpha}}  \right\},
   \end{align*} where $\ee$ and $A$ are constants given by Lemma \ref{cly:SH14_1}, $Q$ is $Q_g$.
\end{enumerate}
  Then we have the inequality
   \begin{align}
     Q(t) < \frac{2}{\sqrt{Q(0)^{-2}+t}}, \quad \forall \; t \in \left[-Q(0)^{-2}, 0 \right].
   \label{eqn:SH15_5}
   \end{align}
    Consequently, we have
    \begin{align}
      F(0) \geq \frac{1}{5Q(0)^2}.
    \label{eqn:SH15_6}
    \end{align}
   \label{prn:HG25_1}
 \end{proposition}

 \begin{proof}
   Let  $\frac{2}{\sqrt{Q(0)^{-2}+t}}$ be a barrier function. Clearly, this function
   bounds $Q(t)$ when $t=-Q(0)^{-2}$ and $t=0$. Let $I$ be the collection of
   $t \in [-Q(0)^{-2}, 0]$ such that
   \begin{align*}
     Q(t) \geq \frac{2}{\sqrt{Q(0)^{-2}+t}}.
   \end{align*}
   We want to show that $I$ is empty.

   Suppose $I \neq \emptyset$. Clearly, $I$ is   bounded and  closed.
      Let $\underline{t}$ be the infimum of $I$. Then we have
       $\underline{t} \in I$ and
\beq
     Q(\underline{t})= \frac{2}{\sqrt{Q(0)^{-2}+\un t}}. \label{eqA004}
\eeq

   \begin{claim}
     For every $t \in I$, we have
     \begin{align}
       \sup_{s \in [t-Q(t)^{-2}, t]} Q(s) < 2 Q(t).
       \label{eqn:SH15_1}
     \end{align}
     \label{clm:SH15_1}
   \end{claim}
\begin{proof}
   First, we show that (\ref{eqn:SH15_1}) holds at $\underline{t}$. In fact, by the definition of $\un t$ for any
   $s\in [\un t-Q(\un t)^{-2}, \un t]$ we have
   \beqs
   Q(s)&<&\frac 2{\sqrt{Q(0)^{-2}+s}}\leq \frac 2{\sqrt{Q(0)^{-2}+\un t-Q(\un t)^{-2}}}= \frac 2{\sqrt{3}}Q(\un t),
   \eeqs where we used (\ref{eqA004}) in the last equality. Thus, (\ref{eqn:SH15_1}) holds at $\underline{t}$.

We next define
   \begin{align*}
     t_0 \triangleq \sup \Big\{ t\,\Big|\, \textrm{Every} \; s \in I \cap  [\underline{t}, t) \; \textrm{satisfies (\ref{eqn:SH15_1})}  \Big \}.
   \end{align*}
   Clearly, $t_0 \in I$.  In order to prove the Claim, we have  to show that $t_0=\un{t}$, which denotes the supreme of $I$.
   Actually, at time $t_0$,  one of following cases must appear according to the definition of $t_0$.

   \textit{Case 1.}  $\displaystyle  \sup_{s \in [t_0-Q(t_0)^{-2}, t_0]} Q(s) = 2 Q(t_0)$.

   \textit{Case 2.}  $\displaystyle Q(t_0)=\frac{2}{\sqrt{Q(0)^{-2}+t_0}}$.

\noindent We will show that both cases will never happen if $t_0 <\un{t}$.

 (1).  If Case 1 happens, then for some $s_0 \in [t_0-Q(t_0)^{-2}, t_0]$  we have
 \beq
     Q(s_0)=2Q(t_0). \label{eq:A005}
\eeq
We now show that $s_0\in I$. In fact, since $t_0\in I$, we have
\beq
Q(s_0)=2Q(t_0)\geq \frac {4}{\sqrt{Q(0)^{-2}+t_0}}\geq \frac 4{\sqrt{
Q(0)^{-2}+s_0+Q(t_0)^{-2}}}.
\eeq This implies that
$$Q(s_0)=2Q(t_0)\geq \frac {2\sqrt{3}}{\sqrt{Q(0)^{-2}+s_0}}\geq \frac {2 }{\sqrt{Q(0)^{-2}+s_0}}$$
Thus, we have $s_0\in I. $

  By the assumption of $t_0$, we have
   \begin{align*}
     \sup_{s \in [s_0-Q(s_0)^{-2}, s_0]} Q(s) \leq 2Q(s_0).
   \end{align*}
   Then $Q$ drops from $Q(s_0)=2Q(t_0)$ to $Q(t_0)$ in a time period $|t_0-s_0|$. By Lemma \ref{cly:SH14_1} we have
   \begin{align*}
     \int_{s_0}^{t_0} \,O_g^{\alpha} Q^{2-\alpha} dt > \frac{\epsilon}{A}.
   \end{align*}
   Note that $|t_0-s_0|<Q(t_0)^{-2}$ and $O_g(t)\leq 1, Q\leq 2Q(t_0)$ for $t\in [s_0, t_0]$, we then obtain
   \begin{align*}
     2^{2-\alpha} Q(t_0)^{-\alpha} > \frac{\epsilon}{A}.
   \end{align*}
  Combining this with the inequality $Q(t_0) \geq \frac{2}{\sqrt{Q(0)^{-2}+t_0}}$, we have
   \begin{align*}
     Q(0) < 2^{\frac{2(1-\alpha)}{\alpha}} \left( \frac{A}{\epsilon} \right)^{\frac{1}{\alpha}},
   \end{align*}
   which contradicts  the choice of $Q(0)$.  Therefore, Case 1 cannot happen.

 (2).  If Case 2 happens, by the definition of $t_0$  we  have
   \begin{align*}
     \frac{d}{dt}Q \leq \frac{d}{dt} \left( \frac{2}{\sqrt{Q(0)^{-2}+t}} \right)
   \end{align*}
   at time $t_0$. Recall that $Q(t_0)=\frac{2}{\sqrt{Q(0)^{-2}+t_0}}$ and
   \begin{align}
     \frac{d}{dt} Q \geq -AO_g^{\alpha} Q^{3-\alpha}
   \label{eqn:SH15_3}
   \end{align}
   at time $t_0$.   It follows that
   \begin{align*}
     -A O(t_0)^{\alpha} Q(t_0)^{3-\alpha} \leq -\frac{1}{\left( Q(0)^{-2} +t_0\right)^{\frac{3}{2}}},
   \end{align*}
   which implies that
   \begin{align}
      Q(0) \leq A^{\frac{1}{\alpha}} \cdot 2^{\frac{3-\alpha}{\alpha}},
   \label{eqn:SH15_4}
   \end{align}
   which contradicts  the choice of $Q(0)$. Thus, Case 2 cannot happen and the Claim is proved.  \\

   \end{proof}

    Since $I$ is closed, we have the supreme   $\un{t} \in I$.  It is clear that $\un{t}<0$ since
   \begin{align*}
     Q(0) < \frac{2}{\sqrt{Q(0)^{-2}+0}}.
   \end{align*}
   Then at time $\un{t}$, we have
     \begin{align*}
     \frac{d}{dt}Q(t) \leq \frac{d}{dt} \left( \frac{2}{\sqrt{Q(0)^{-2}+t}} \right).
   \end{align*}
   Combining this with (\ref{eqn:SH15_3}), we  get the inequality (\ref{eqn:SH15_4}) again. So this contradiction implies that
   our assumption $I \neq \emptyset$ is wrong. The Proposition is  proved.
 \end{proof}

Using Proposition \ref{prn:HG25_1}, we can estimate the curvature scale when the curvature at some time is not large.
 \begin{proposition}\label{theo:reg2}
 Suppose $\left\{ (M^n, g(t)), -2 \leq t \leq 0 \right\}$ is a  Calabi flow solution satisfying
  \begin{itemize}
   \item  the scalar curvature $|S|(t) \leq 1$ for all $t\in [-2, 0]$.
   \item  the curvature tensor satisfies
   $$Q(0) \leq \max \left\{2^{\frac{3-\alpha}{\alpha}} A^{\frac{1}{\alpha}},  2^{\frac{2(1-\alpha)}{\alpha}} \left( \frac{A}{\epsilon}\right)^{\frac{1}{\alpha}}  \right\}. $$
 \end{itemize}
 Then there is a constant $c_0=c_0(n, \ee, A)>0$ such that
 \begin{align}
      F(0) \geq c_0(n, \ee, A).
      \label{eqn:SH14_1}
 \end{align}
  Consequently, we have
 \begin{align*}
     \sup_{x \in M} \left|\nabla^k \Rm \right|(x,0) \leq C(n, \ee, A,  k).
 \end{align*}
 \label{prn:HG25_2}
 \end{proposition}
  \begin{proof}
Let
  \beq
 L=2\max \left\{2^{\frac{3-\alpha}{\alpha}} A^{\frac{1}{\alpha}},  2^{\frac{2(1-\alpha)}{\alpha}} \left( \frac{A}{\epsilon}\right)^{\frac{1}{\alpha}}  \right\}
  \eeq
and we define
 \begin{align*}
   s_1 =
   \begin{cases}
     \sup \left\{ t \,|\, Q(t)=  L, -1 \leq t \leq 0 \right\},  & \textrm{if} \; \left\{ t\,|\,Q(t)=L, -1 \leq t \leq 0 \,\right\} \neq \emptyset, \\
     -1, &\textrm{if} \; \left\{ t\,|\,Q(t)=L, -1 \leq t \leq 0\, \right\} =\emptyset.
   \end{cases}
 \end{align*} If $s=-1$, then the theorem holds. Otherwise, by Proposition \ref{prn:HG25_1}  the curvature satisfies
\beq
Q(t)\leq \sqrt{5}L, \quad t\in \left[s_1-\frac 1{5L^2}, s_1 \right].
\eeq On the other hand, we have
$Q(t)\leq L$ for any $t\in [s_1, 0]$. Therefore, $Q(t)\leq \sqrt{5}L$
holds for any $t\in \left[s_1-\frac 1{5L^2}, 0 \right]$. Since the interval  $\left[-\frac 1{5L^2}, 0\right]$ is contained in $\left[s_1-\frac 1{5L^2}, 0\right]$,
we get
$$Q(t)\leq \sqrt{5}L, \quad t\in \left[-\frac 1{5L^2}, 0\right]. $$
The Proposition is proved.
\end{proof}


\def\hr{{\mathrm {hr}}}\subsection{Harmonic scale}

In Proposition~\ref{prn:HG25_1} and  Proposition~\ref{prn:HG25_2}, we prove the ``stability" of the curvature scale under the scalar bound condition.
However, this condition is in general not available.  We observe that Calabi energy is scaling invariant  for complex dimension 2.
For this particular dimension, scalar bound condition can more or less be replaced by Calabi energy small, whenever collapsing does not happen.
For the purpose to rule out collapsing, we need a more delicate scale, which is the harmonic scale introduced in this subsection.

\begin{definition}(cf. \cite{[And]}\cite{[AC]})Let $(M^n, g)$ be an $n$-dimensional  Riemannian manifold. Given $p\in (n, \infty)$ and $Q>1$, the $L^{1, p}$ harmonic radius $\hr(x, g)$ at the point $x\in M$ is the largest number $r_0$ such that on the geodesic ball $B=B_x(r_0)$ of radius $r_0$ in $(M, g)$, there is a harmonic coordinate chart $U=\{u_i\}_{i=1}^n:B\ri \RR^n$, such that the metric tensor $g_{ij}=g(\pd {}{u_i},
\pd {}{u_j})$ satisfies
\beqs
Q^{-1}(\dd_{ij})\leq (g_{ij})\leq Q(\dd_{ij}), \quad r_0^{1-\frac np}\|\p g_{ij}\|_{L^p}
\leq Q-1.
\eeqs
The harmonic radius $\hr_g(M)$ is defined by $$\hr_g(M)=\inf_{x\in M}\hr(x, g).$$
\end{definition}

For the harmonic radius, Anderson-Cheeger showed the following result.

\begin{lemma}
\label{lem:harmonic}(cf.~\cite{[AC]})Fix $Q>1.$ Let $(M_i, g_i)$ be a sequence of Riemannian manifolds
which converges strongly in $L^{1, p}$ topology to a limit $L^{1, p}$ Riemannian manifold $(M, g)$. Then
$$\hr_g(M)=\lim_{i\ri\infty}\hr_{g_i}(M_i). $$
Moreover, for any $x_i\in M_i$ with $x_i\ri x\in M$ we have
$$\hr(x, g)=\lim_{i\ri \infty}\hr(x_i, g_i). $$

\end{lemma}

Next, we introduce the harmonic scale which will be used in the convergence of Calabi flow on K\"ahler surfaces.

\begin{definition}Let $(M, g)$ be a Riemannian manifold. The harmonic scale $H_g(M)$ is
the supreme of $r$ such that
$$
  \max_M |\Rm| < r^{-2}, \quad
   \hr(x, g) > r, \quad \forall \; x \in M.
$$
In other words, the harmonic scale of $(M,g)$ is defined by
\begin{align*}
   H_g(M)= \min \left\{ \left(\sup_M |\Rm| \right)^{-\frac{1}{2}}, \quad \hr_g(M) \right\}.
\end{align*}

\end{definition}

\begin{lemma}
There is a universal small constant $\epsilon$ with the following properties.

   Suppose $\{(M^2,g(t)), -K \leq t \leq 0\}$ is a Calabi flow solution on a compact K\"ahler surface, $K \geq 2$.
   Then for every $t \in [-1, 0]$, we have
   \begin{align*}
      H_{g(t)}(M) \geq \frac{1}{2}
   \end{align*}
    whenever  $H_{g(0)}(M) \geq 1$ and $ Ca(-K) -Ca(0)<\epsilon$.
\label{lem:900}
\end{lemma}
\begin{proof}
 We argue by contradiction.
 Suppose the statement were wrong.  Then we can find a sequence of Calabi flows $\{(M_i, g_i(t)), -K_i \leq t \leq 0\}(K_i>2)$ violating the statement with
   \begin{align*}
     Ca_{g_i}(-K_i)-Ca_{g_i}(0)<\epsilon_i, \quad  \epsilon_i \to 0.
   \end{align*}

   Let $\{(M,g(t)), -K \leq t \leq 0\}$ be one of such flows. We shall truncate a critical time-interval from this flow.
   Check if there is a time such that $H(t)<\frac{1}{2}$ in $[-1,0]$. If no, stop. Otherwise, choose the first time $t$ such that $H(t)=\frac{1}{2}$ and
   denote it by $t_1$.  Then check the interval $[t_1-H^4(t_1), t_1]$ to see if there is a time such that $H(t) \leq  \frac{1}{2} H(t_1)$.
   If no such time exists, we stop.  Otherwise, repeat the process.  Note that
   \begin{align*}
     &H(t_k)=\frac{1}{2^k}, \\
     & |t_{k+1}-t_k| \leq \frac{1}{16^k}, \\
     & |t_k| \leq 1+\frac{1}{16} +\frac{1}{16^2}+\cdots < \frac{16}{15}.
   \end{align*}
   This process happens in a compact smooth space-time $M \times [-2,0]$ with bounded geometry. In particular,  the harmonic scale is bounded.
   After each step, the harmonic scale decrease one half.
    Therefore, it must stop in finite steps.  Suppose it stops at $(k+1)$-step.
   Therefore, for some time $t_{k+1} \in [t_k-H^4(t_k), t_k]$, we have
   \begin{align*}
     &H(t_{k+1})=\frac{1}{2}H(t_k), \\
     &H(t) \geq \frac{1}{2} H(t_{k+1}), \quad \; \forall \; t \in \left[ t_{k+1}-H^4(t_{k+1}), t_{k+1} \right].
   \end{align*}
   Denote $r_k=H(t_k)$ and let $\tilde{g}(t)=r_k^{-2}g(t_k+r_k^4 t), \; s=r_k^{-4}(t_{k+1}-t_k)\in[-1,0)$.
   Then for the flow $\tilde{g}$, we have
   \begin{align*}
    &H(M,\tilde{g}(0))=1, \\
    &H(M,\tilde{g}(s))=\frac{1}{2}, \\
    &H(M,\tilde{g}(t)) \geq \frac{1}{2}, \quad \forall \; t \in [s,0], \\
    &H(M,\tilde{g}(t)) \geq \frac{1}{32}, \quad \forall \; t \in \left[s-\frac{1}{16},s \right], \\
    &Ca(M, \tilde{g}(-2))-Ca(M, \tilde{g}(0))<\epsilon.
   \end{align*}
   Now for each flow $g_i$, we rearrange the base point and rescale the flow according to the above arrangement. Denote the new flows
   by $\{(M_i, \tilde{g}_i(t)), -1 \leq t \leq 0\}$.  Then the above equations hold for each $\tilde{g}_i$ with some $s_i \in [-1,0)$ and $\epsilon_i \to 0$.
   Let $x_i$ be the point where $H(\tilde{g}_i, s_i)$ achieves value. In other words, we have
   \begin{align}
     &H(M, \tilde{g}_i(s_i))=\frac{1}{2}, \label{eqn:HB28_1} \\
     &hr(x_i,\tilde{g}_i(s_i))=\frac{1}{2}, \quad \textrm{or} \quad  |\Rm|_{\tilde{g}_i(s_i)}(x_i) =4.  \label{eqn:HB28_2}
   \end{align}
   Let $\un{s}$ be the limit of $s_i$. Then on $(\un{s}-\frac{1}{16},0)$, we have uniform bound of $H$ when time is uniformly bounded away from
   $\un{s}-\frac{1}{16}$.  Then we have bound of curvature, curvature higher derivatives, injectivity radius, etc.
  Therefore, we can take smooth convergence on time interval $(\un{s}-\frac{1}{16}, 0)$,
   \begin{align*}
      \left\{ \left(M,x_i,\tilde{g}_i(t) \right), \un{s}-\frac{1}{16} <t \leq 0 \right\} \longright{Cheeger-Gromov-C^{\infty}}
       \left\{\left(\un{M},\un{x},\un{g}(t) \right), \un{s}-\frac{1}{16} <t \leq 0 \right\},
   \end{align*}
   and the Calabi energy of the limit metric $\un g(t)$ is static for all $t\in   (\un s-\frac 1{16}, 0)$. Note that for each fixed compact set $\Omega \subset \underline{M}$,
   the integral of $|\nabla \nabla S|^2$ on $M \times [\un{s}-\frac{1}{16}, 0]$ is dominated by
   \begin{align*}
     \lim_{i \to \infty}  Ca(\tilde{g_i}(-2))-Ca(\tilde{g_i}(0))=0.
   \end{align*}
   Therefore,  on the limit flow $\un g(t)$, we have $|\nabla \nabla S| \equiv 0$.  Every $\un g(t)$ is an extK metric and
   $\un g(t)$ evolves by the automorphisms group generated by $\nabla S$. In particular, the intrinsic Riemannian geometry does not evolve along the flow.
   From $t=\underline{s}$ to $t=0$, suppose the generated automorphism is $\varrho$.  Clearly, $\varrho$ is the limit of diffeomorphisms $\varrho_i$, which is the integration of  the real vector field $\nabla S_i$ from time $t=s_i$  to time $t=0$.

   Note that at time $\un{s}$ and $0$, we have a priori bound for all high curvature derivatives of curvature.
   By Lemma \ref{lem:harmonic}, the harmonic radius is continuous in the smooth convergence.
   Note that at time $t=0$,  harmonic scale is $1$, which implies that
   \begin{align*}
     |\Rm|_{\tilde{g}_i(0)}(x) \leq 1,  \quad \hr(x_i, \tilde{g}_i(0))\geq 1.
   \end{align*}
   Therefore, we have
   \begin{align*}
   &\hr(\un{x},\un{g}(0))=\lim_{i \to \infty} \hr(x_i, \tilde{g}_i(0)) \geq 1, \\
   &\hr(\un{x},\un{g}(\un{s}))=\lim_{i \to \infty} \hr(x_i, \tilde{g}_i(s_i)).
   \end{align*}
   Since $\un{g}$ evolves by automorphisms, we have
   \begin{align*}
   \lim_{i \to \infty} \hr(x_i, \tilde{g}_i(s_i))=hr(\underline{x}, \underline{g}(s))=hr(\varrho(\underline{x}), \underline{g}(0))=
   \lim_{i \to \infty} \hr(\varrho_i(x_i), \tilde{g}_i(0)) \geq 1.
   \end{align*}
   On the other hand, it is clear that
   \begin{align*}
     &\quad \lim_{i \to \infty} |\Rm|_{\tilde{g}_i(s_i)}(x_i)
     =|\Rm|_{\un{g}(\un{s})}(\un{x})=|\Rm|_{\un{g}(0)}(\varrho(\un{x}))
     =\lim_{i \to \infty} |\Rm|_{\tilde{g}_i(0)}(\varrho_i(x_i)) \leq 1.
   \end{align*}
   Therefore, for large $i$, we have
$$
    \hr(x_i, \tilde{g}_i(s_i))>\frac{3}{4}, \quad
     \sup_{B_{\tilde{g}_i(s_i)}(x_i,\frac{1}{2})} |\Rm|_{\tilde{g}_i(s_i)}<\frac{4}{3},
$$
   which contradicts   (\ref{eqn:HB28_2}). The lemma is proved.
\end{proof}

\subsection{Backward regularity improvement}
We now can summarize the main results in Section~\ref{sec:regularity} as the following backward regularity improvement theorems.

\begin{theorem}
There is a $\delta=\delta(B, c_0)$ with the following properties.

Suppose $T \geq T_0$ and $\{(M^2, \omega(t), J),  0 \leq t \leq T\}$ is a Calabi flow solution satisfying
\begin{align*}
&Ca(0)-Ca(T)<\epsilon, \\
&\sup_{M} |Rm|(\cdot, T) \leq B,\\
&inj(M, g(T)) \geq c_0,
\end{align*}
where $\epsilon$ is the universal small constant in Lemma~\ref{lem:900}.
Then we have
  \begin{align*}
     \sup_{M \times [T_0-\delta, T_0]}  |\nabla^l Rm| \leq C_l,  \quad \forall \; l \in \Z^+ \cup \{0\}.
  \end{align*}
\label{thm:HL29_1}
\end{theorem}

\begin{proof}
 Note that $H_{g(T)}(M)$ is uniformly bounded from below. Therefore, up to rescaling, we can apply Lemma~\ref{lem:900} to obtain a $\delta$ such that
 $H_{g(t)}(M)$ is uniformly bounded from below whenever $t \in [T-2\delta, T]$.   Then the statement follows from the application of Theorem~\ref{thm:GA04_JS}.
\end{proof}

\begin{theorem}
There is a $\delta=\delta(n,T_0,A,B)$ with the following properties.

Suppose $T \geq T_0$ and $\{(M^n, \omega(t), J),  0 \leq t \leq T\}$ is a Calabi flow solution satisfying
\begin{align*}
&\sup_{M \times [0, T]} |S| <A, \\
&\sup_{M} |Rm|(\cdot, T) \leq B.
\end{align*}
Then we have
  \begin{align*}
     \sup_{M \times [T_0-\delta, T_0]}  |\nabla^l Rm| \leq C_l,  \quad \forall \; l \in \Z^+ \cup \{0\}.
  \end{align*}
\label{thm:HL29_2}
\end{theorem}

The proof of Theorem~\ref{thm:HL29_2} is nothing but a mild application of Proposition~\ref{prn:HG25_1} and Proposition~\ref{prn:HG25_2}, with loss of accuracy.
The reason for developing Proposition~\ref{prn:HG25_1} and Proposition~\ref{prn:HG25_2} with more precise statement is for the later use in Section~\ref{subsec:finite}, where we
study the blowup rate of Riemannian curvature tensors.    Note that Theorem~\ref{thm:HL29_2} deals with collapsing case also.
If we add a non collapsing condition at time $T$, then the scalar curvature bound in Theorem~\ref{thm:HL29_2} can be replaced by a uniform bound of $\norm{S}{L^p}$ for
some $p>n$.  This was pointed out by S.K. Donaldson~\cite{Don14}.

\begin{remark}
  All the quantities in Theorem~\ref{thm:HL29_1} and Theorem~\ref{thm:HL29_2} are geometric quantities, and consequently are invariant under the action of diffeomorphisms.
  Therefore, if we transform the Calabi flow by  diffeomorphisms, then all the estimates in Theorem~\ref{thm:HL29_1} and Theorem~\ref{thm:HL29_2} still hold.  In particular, they hold
  for the modified Calabi flow(c.f. Definition~\ref{dfn:HL27_1}) and the complex structure Calabi flow(c.f. equation (\ref{eqn:HL31_1})).
\label{rmk:HL30_1}
\end{remark}

\section{Asymptotic behavior of the Calabi flow}
\label{sec:longtime-behavior}

\subsection{Deformation of the modified Calabi flow around extK metrics}

In this subsection, we fix the underlying complex manifold and evolve the Calabi flow in a fixed K\"ahler class.

Let $\omega$ be an extK metric, $X$ to be the extremal vector field defined by the extremal K\"ahler metric $\om$, i.e.,
\begin{align}
X=g^{i\bar k}\frac{\partial S}
{\partial z^{\bar k}}\frac{\partial}
{\partial z^{i}}.
\label{eqn:GA06_1}
\end{align}
It is well known that $X$ is a holomorphic vector field, due to the work of Calabi~\cite{[Cal1]}.
Recall that for every holomorphic vector field $V$, the Futaki invariant is defined to be
\begin{align}
\Fut(V,[\om]) \triangleq \int_M \,V(f)\;\om^n
\label{eqn:HL25_2}
\end{align}
where $f$ is the normalized scalar potential defined by
\begin{align*}
   \Delta f =S-\underline S, \quad \int_M e^f\om^n=\mathrm{vol}(M).
\end{align*}
Note that (\ref{eqn:HL25_2}) is well-defined since the right hand side of (\ref{eqn:HL25_2}) depends only on the K\"ahler class $[\omega]$.
Let $V=X$, then we have
\begin{align*}
  \Fut(X,[\omega])=\int_M (S-\underline{S})^2 \omega^n=Ca(\omega).
\end{align*}
According to \cite{MR2471594}, in the class $[\omega]$, the minimal value of the Calabi energy is achieved at $\omega$.  In other words, for every smooth metric
$\omega_{\varphi}$, we have
\begin{align*}
  Ca(\omega_{\varphi}) \geq Ca(\omega),
\end{align*}
with equality holds if and only if $\omega_{\varphi}$ is also an extK metric.
Note that $\omega_{\varphi}$ is extremal if and only if $\varrho_0^*  \omega_{\varphi}=\omega$ for some $\varrho_0 \in Aut_0(M,J)$,
by the uniqueness theorem of extK metrics(c.f.~\cite{[ChenTian]},~\cite{arXiv:1405.0401},~\cite{arXiv:1409.7896}).

In the K\"ahler class $[\omega]$, consider a smooth family of K\"ahler metrics $g(t)$ satisfying
\begin{align}
   \frac{\partial}{\partial t} g_{i\bar{j}}= S_{,i\bar{j}} + L_{Re(X)} g_{i\bar{j}},
\label{eqn:HL25_3}
\end{align}
where $Re(X)$ is the real part of the holomorphic vector field $X$ defined in (\ref{eqn:GA06_1}).
The above flow was considered in \cite{[HZ]} and Section 3.2 of \cite{[LZ]}.

\begin{definition}
  Equation (\ref{eqn:HL25_3}) is called the modified Calabi flow equation. Correspondingly,   the functional
  $Ca(\omega_{\varphi})-Ca(\omega)$ is called the modified Calabi energy.
\label{dfn:HL27_1}
\end{definition}

The space $\mathcal{H}$ (c.f. equation (\ref{eqn:HL30_1})) has an infinitely dimensional Riemannian symmetric space structure,
as described by Donaldson \cite{[Don96]},  Mabuchi~\cite{MR0909015} and Semmes \cite{[Semmes]}.
Every two metrics $\omega_{\varphi_1}, \omega_{\varphi_2}$ can be connected by a weak $C^{1,1}$-geodesic, by the result of Chen~\cite{Chen1}.
Therefore, $\mathcal{H}$ has a metric induced from the geodesic distance $d$, which plays an important role in the study of the Calabi flow.
For example,  the Calabi flow decreases the geodesic distance in $\mathcal{H}$(c.f.~\cite{MR1969662}).
Furthermore, by the invariance of geodesic distance up to automorphism action,
the modified Calabi flow also decreases the geodesic distance.
However, $d$ is too weak for the purpose of improving regularity.
Even if we know that $d(\omega_{\varphi}, \omega)$ is very small, we cannot
obtain too much information of $\omega_{\varphi}$.
For the convenience of improving regularity, we introduce an auxiliary function $\hat{d}$ on $\mathcal{H}$.

\begin{definition}
   For each $\omega_{\varphi}$ in the class $[\omega]$, define
   \begin{align*}
        \hat{d}(\omega_{\varphi}) \triangleq    \inf_{\varrho \in Aut_0(M,J)} \norm{\varrho^* \omega_{\varphi} -\omega}{C^{k,\frac{1}{2}}}.
   \end{align*}
\label{dfn:HL25_1}
\end{definition}
Note that $\hat{d}$ is not really a distance function.    The advantage of $\hat{d}$ is that if $\hat{d}$ is very small, then one can choose an automorphism $\varrho$ such that $\varrho^* \omega_{\varphi}$ is around the extK metric $\omega$, in the $C^{k, \frac{1}{2}}$-norm.
Then regularity improvement of $\varrho^* \omega_{\varphi}$ becomes possible.
Note that in the K\"ahler class $[\omega]$, a metric form $\omega_{\varphi}$ is extremal if and only if  the modified Calabi energy vanishes,
in light of the result of Chen in~\cite{MR2471594}.  Then it follows from the uniqueness of extemal metrics that $\varrho_0^*  \omega_{\varphi}=\omega$ for some
$\varrho_0 \in Aut_0(M,J)$.   By the definition of $\hat{d}$, we have
\begin{align*}
  \hat{d}(\omega_{\varphi})=\inf_{\varrho \in Aut_0(M,J)} \norm{\varrho^* \omega_{\varphi} -\omega}{C^{k,\frac{1}{2}}} \leq \norm{\varrho_0^* \omega_{\varphi} -\omega}{C^{k,\frac{1}{2}}}=0.
\end{align*}
In short, $Ca(\omega_{\varphi})-Ca(\omega)=0$ implies that $\hat{d}(\omega_{\varphi})=0$.  The following lemma indicates that there is an almost version of this phenomenon.

\begin{lemma}
For each $\epsilon>0$, there is a $\delta=\delta(\omega, \epsilon)$ with the following property.

If $\varphi \in \mathcal{H}$ satisfies
 \begin{align*}
   \norm{\omega_{\varphi}-\omega}{C^{k+1,\frac{1}{2}}} <1, \quad    Ca(\omega_{\varphi}) - Ca(\omega)<\delta,
 \end{align*}
 then $\hat{d}(\omega_{\varphi})<\epsilon$.
\label{lma:HL24_1}
\end{lemma}

\begin{proof}
 For otherwise, there is an $\epsilon_0>0$ and a sequence of $\varphi_i \in \mathcal{H}$ satisfying
 \begin{align}
     &\norm{\omega_{\varphi_i}-\omega}{C^{k+1,\frac{1}{2}}} <1,  \quad Ca(\omega_{\varphi_i}) - Ca(\omega)<\delta_i \to 0,  \label{eqn:HL24_1}\\
     & \hat{d}(\omega_{\varphi_i}) \geq \epsilon_0.   \label{eqn:HL24_2}
 \end{align}
 Then we can assume that $\omega_{\varphi_i}$ converges to $\omega_{\varphi_{\infty}}$ in the $C^{k,\frac{1}{2}}$-topology, which gives rise to $C^{k,\frac{1}{2}}$-metric with Calabi energy the same as $Ca(\omega)$.
 By the uniqueness theorem(c.f.~\cite{[ChenTian]},~\cite{arXiv:1405.0401},~\cite{arXiv:1409.7896}), we can find an automorphism $\varrho_{\infty} \in Aut_0(M,J)$ such that
 \begin{align*}
      \varrho_{\infty}^* \omega_{\varphi_{\infty}}=\omega.
 \end{align*}
 Note that $\varrho_{\infty}$ is automatically smooth. Then we have
 \begin{align*}
    \varrho_{\infty}^* \omega_{\varphi_i}  \longright{C^{k,\frac{1}{2}}} \omega,
 \end{align*}
 which contradicts (\ref{eqn:HL24_2}).
\end{proof}

There is a modified Calabi flow version of the  short-time existence theorem of Chen-He(c.f.~\cite{[ChenHe1]}).

\begin{lemma}
There is a $\xi_0=\xi_0(\omega)$ with the following properties.

Suppose $ \norm{\omega_{\varphi}-\omega}{C^{k,\frac{1}{2}}}<\xi_0$.  Then the modified Calabi flow starting from $\omega_{\varphi}$ exists on time interval $[0,1]$ and we have
\begin{align*}
   \sup_{\frac{1}{2} \leq t \leq 1} \norm{\omega_{\varphi(t)}-\omega}{C^{k+1,\frac{1}{2}}} <1.
\end{align*}
\label{lma:HL24_2}
\end{lemma}

Now we fix $\xi_0$ in Lemma~\ref{lma:HL24_2} and define $\delta_0=\delta_0(\omega, \xi_0)$ as in Lemma~\ref{lma:HL24_1}.

\begin{lemma}
Suppose $\eta_0< \xi_0$ is small enough such that $ Ca(\omega_{\varphi})<\delta_0$
for every $\varphi \in \mathcal{H}$ satisfying $\norm{\omega_{\varphi}-\omega}{C^{k,\frac{1}{2}}}<\eta_0$.
Then the modified Calabi flow starting from $\omega_{\varphi}$ has global existence whenever
$\norm{\omega_{\varphi}-\omega}{C^{k,\frac{1}{2}}}<\eta_0$.
\label{lma:HL24_3}
\end{lemma}

\begin{proof}
Clearly, the flow exists on $[0,1]$ and $\hat{d}(\omega_{\varphi(1)})<\xi_0$.
We continue to use induction to show that the flow exists on $[0, N]$ for each positive integer $N$ and $\hat{d}(\omega_{\varphi(N)})<\xi_0$.

Suppose the statement holds for $N$. Then at time $N$, we have $\hat{d}(\omega_{\varphi(N)})<\xi_0$. Therefore, we can find an automorphism
$\varrho_N$ such that
\begin{align*}
   \norm{\varrho_N^* \omega_{\varphi(N)}-\omega}{C^{k,\frac{1}{2}}}<\xi_0.
\end{align*}
Then the modified Calabi flow starting from $\varrho_N^* \omega_{\varphi(N)}$ exists for another time period with length $1$, in light of Lemma~\ref{lma:HL24_2}.
Moreover, we have the $C^{k+1,\frac{1}{2}}$-bound of the metric at the end of this time period.
Note that the modified Calabi energy is monotonelly decreasing along the flow. Therefore, at the end of the time period $1$, one can apply Lemma~\ref{lma:HL24_1} to obtain
that $\hat{d}<\xi_0$.
 Since the intrinsic geometry and $\hat{d}$ does not change under the automorphism action, it is clear that the modified Calabi flow starting from  $\omega_{\varphi(N)}$ exists for another time length $1$
 with proper geometric bounds.
 Therefore, $\omega_{\varphi(t)}$ are well defined smooth metrics for $t \in [N, N+1]$ and $\hat{d}(\omega_{\varphi(N+1)})<\xi_0$.
\end{proof}

We observe that the automorphisms $\varrho_i$ defined in the proof of Lemma~\ref{lma:HL24_3} are uniformly bounded.
Actually, note that the geodesic distance between two modified Calabi flows is non-increasing.
By triangle inequality, we have
\begin{align*}
  d(\omega, \omega_{\varphi})
  &\geq
  d(\omega, \omega_{\varphi(i)})
  \geq d(\omega, (\varrho_i^{-1})^\ast \omega) -d((\varrho_i^{-1})^\ast \omega, \omega_{\varphi(i)})\\
  &=d(\omega,(\varrho_i^{-1})^\ast \omega)-d(\omega,\varrho_i^\ast \omega_{\varphi(i)}).
\end{align*}
It follows that
\begin{align*}
   d(\omega, (\varrho_i^{-1})^\ast \omega) \leq d(\omega, \omega_{\varphi}) + d(\omega,\varrho_{i}^\ast \omega_{\varphi(i)})< C.
\end{align*}
Therefore, $\varrho_i$ must be uniformly bounded.   Furthermore, since the modified Calabi energy always tends to zero as $t \to \infty$,
by the results of Streets~\cite{[Streets1]} and He~\cite{[He5]}, it is clear that
\begin{align*}
    Ca(\omega_{\varphi(i)})-Ca(\omega_{\varphi(i-1)}) \to 0.
\end{align*}
Applying Lemma~\ref{lma:HL24_1}, we can choose $\xi_i \to 0$ such that
\begin{align*}
   \norm{\varrho_i^* \omega_{\varphi(i)}-\omega}{C^{k,\frac{1}{2}}}<\xi_i \to 0, \quad \textrm{as} \; i \to \infty.
\end{align*}
In particular, we have $d(\omega, \varrho_i^* \omega_{\varphi(i)}) \to 0$ as $i \to \infty$.

\begin{theorem}
Suppose $\varphi \in \mathcal{H}$ satisfying $\norm{\omega_{\varphi}-\omega}{C^{k,\frac{1}{2}}}<\eta_0$.
Then the modified Calabi flow starting from $\omega_{\varphi}$
converges to $\varrho^\ast \om$ for some $\varrho\in  \mathrm{Aut}_0(M,J)$, in the smooth topology of K\"ahler potentials.
\label{thm:HL24_1}
\end{theorem}

\begin{proof}

First, let us show the convergence in distance topology.

Let $\varrho_j$ be the automorphisms defined in the proof of Lemma~\ref{lma:HL24_3}.   From the above discussion,
we see that  $\varrho_j$ are uniformly bounded, by taking subsequence if necessary, we can assume
 $\varrho_j$ converges to $\varrho_{\infty}$.   Then
\begin{align*}
  \lim_{j\rightarrow+\infty}d((\varrho_{\infty}^{-1})^\ast \omega, \omega_{\varphi(j)})= \lim_{j\rightarrow+\infty}d((\varrho_j^{-1})^\ast \omega, \omega_{\varphi(j)})
  =\lim_{j\rightarrow+\infty} d(\omega, (\varrho_j)^\ast \omega_{\varphi(j)})=0.
\end{align*}
Note that $(\varrho_{\infty}^{-1})^\ast \omega$ is also an extremal K\"ahler metric and hence a fixed point of the modified Calabi flow.
By the monotonicity of geodesic distance along the modified Calabi flow, we have
\begin{align*}
   d((\varrho_{\infty}^{-1})^\ast \omega, \omega_{\varphi(t)}) \leq d((\varrho_{\infty}^{-1})^\ast \omega, \omega_{\varphi(j)})
\end{align*}
whenever $t \geq j$.  In the above inequality,  let $t \to \infty$ and then let $j \to \infty$, we obtain
\begin{align}
    \lim_{t \to \infty} d((\varrho_{\infty}^{-1})^\ast \omega, \omega_{\varphi(t)})=0.    \label{eqn:HL25_1}
\end{align}
This means that $\omega_{\varphi(t)}$ converges to $(\varrho_{\infty}^{-1})^\ast \omega$ in the distance topology.

Second, we improve the convergence topology from distance topology in (\ref{eqn:HL25_1}) to smooth topology.

Define $\varrho(t)=\varrho_j$ for each $t \in [j, j+1)$.
In light of the proof of Lemma~\ref{lma:HL24_3},  both $\varrho(t)^* \omega_{\varphi(t)}$ and $\varrho(t)$ are uniformly bounded.  It follows that $\omega_{\varphi(t)}$ are uniformly bounded in each $C^l$-topology.
Let $t_i \to \infty$ be a time sequence such that
\begin{align*}
     \omega_{\varphi(t_i)}   \longright{C^{\infty}}   \omega_{\varphi_{\infty}}, \quad \textrm{as} \quad i \to \infty.
\end{align*}
In view of (\ref{eqn:HL25_1}), we see that $d( \omega_{\varphi_{\infty}}, (\varrho_{\infty}^{-1})^* \omega)=0$, which forces that $\omega_{\varphi_{\infty}}=(\varrho_{\infty}^{-1})^* \omega$.
Since $\{t_i\}$ is arbitrary time sequence such that $ \omega_{\varphi(t_i)}$ converges,  the above discussion actually implies that
\begin{align*}
    \omega_{\varphi(t)}   \longright{C^{\infty}}   (\varrho_{\infty}^{-1})^* \omega, \quad \textrm{as} \quad t \to \infty.
\end{align*}
Let $\varrho=\varrho_{\infty}^{-1}$. Then the proof of the Theorem is complete.
\end{proof}

The  convergence in Theorem \ref{thm:HL24_1} could be as precise as ``exponential" if we have further conditions on $\omega$ or $\omega_{\varphi}$.

\begin{theorem}[\cite{[ChenHe1]}, \cite{[HZ]}]
Same conditions as those in Theorem~\ref{thm:HL24_1}. Suppose one of the following conditions is satisfied.
\begin{itemize}
\item $\omega$ is a cscK metric.
\item $\omega_{\varphi}$ is $G$-invariant, where $G$ is a maximal compact subgroup of $Aut_0(M,J)$.
\end{itemize}
Then the modified Calabi flow starting from $\omega_{\varphi}$ converges  to $\varrho^* \omega$ exponentially fast in the smooth topology of K\"ahler potentials.
In other words,  for each positive integer $l$, there are constants $C, \theta$ depending on $l$ such that
$$|\vphi(t)-\vphi_\infty|_{C^{l}(g)}\leq C_l\cdot e^{-\theta t}, \forall t\geq 0.$$
\label{thm:HL24_exp}
\end{theorem}

\subsection{Convergence of K\"ahler potentials}
\label{subset:potential}

The key of this subsection are the following regularity-improvement properties.

\begin{proposition}
 Suppose $\{(M^2, \omega_{\varphi(t)}, J),  0 \leq t < \infty\}$ is a modified Calabi flow solution satisfying
 \begin{align}
   \norm{\omega_{\varphi(t_0)} -\omega}{C^{k, \frac{1}{2}}}<\delta_0  \label{eqn:HL27_1}
 \end{align}
 for some $t_0 \geq 1$.   Then for each positive integer $l \geq k$, there is a constant $C_l$ depending on $l$ and $\omega$ such that
 \begin{align}
   \norm{\omega_{\varphi(t_0)} -\omega}{C^{l, \frac{1}{2}}}< C_l.
 \label{eqn:HL29_1}
 \end{align}
\label{prn:HL27_1}
\end{proposition}

\begin{proof}
The solution of the modified flow (\ref{eqn:HL25_3}) and the Calabi flow differs only by an action of $\varrho$,
which the automorphism generated by $Re(X)$ from time $t_0-1$ to $t_0$, where $X$ is defined in (\ref{eqn:GA06_1}).
Since $X$ is a fixed holomorphic vector field. It follows that $\norm{\varrho}{C^{l,\frac{1}{2}}}$ is uniformly bounded
by $A_l$  for each positive integer $l$.  Therefore, in order to show (\ref{eqn:HL29_1}) for modified Calabi flow,
it suffices to prove it for unmodified Calabi flow.

Clearly, the Riemannian geometry of $\omega_{\varphi(t_0)}$ is uniformly bounded.   By shrinking $\delta_0$ if necessary, we also have
\begin{align*}
  Ca(\omega_{\varphi(t_0-1)})-Ca(\omega_{\varphi(t_0)}) \leq Ca(\omega_{\varphi(t_0)}) -Ca(\omega) <\epsilon,
\end{align*}
where $\epsilon$ is the small constant in Lemma~\ref{lem:900}.
Therefore, Theorem~\ref{thm:HL29_1} applies and we obtain  uniform $|\nabla_\vphi^l Rm(\om_\vphi)|_{g_\vphi}$ bound for each  nonnegative integer $l$.
Then (\ref{eqn:HL29_1}) follows from regularity improvement of Monge-Ampere equation.
Alternatively, one can argue as follows.  Due to the bound of curvature derivatives, one can obtain the metric equivalence for a fixed time peoriod before $t_0$, say on $[t_0-\frac{1}{4}, t_0]$.
Then we see that $\omega_{\varphi(t_0-\frac{1}{4})}$ is uniformly equivalent to $\omega$ and has uniformly bounded Ricci curvature.  This forces that $\omega_{\varphi(t_0-\frac{1}{4})}$
has uniform $C^{1,\frac{1}{2}}$-norm, due to Theorem 5.1 of Chen-He~\cite{[ChenHe1]}.
Consequently, (\ref{eqn:HL29_1}) follows from the smoothing property of the Calabi flow(c.f. the proof of Theorem 3.3 of~\cite{[ChenHe1]}).
\end{proof}

\begin{proposition}
 Suppose $\{(M^n, \omega_{\varphi(t)}, J),  0 \leq t < \infty\}$ is a modified Calabi flow solution satisfying
 \begin{align}
   \norm{\omega_{\varphi(t_0)} -\omega}{C^{k, \frac{1}{2}}}<\delta_0,   \quad \sup_{M \times [t_0-1, t_0]}|S| < A,   \label{eqn:HL27_1}
 \end{align}
 for some $t_0 \geq 1$.   Then for each positive integer $l \geq k$, there is a constant $C_l$ depending on $l$ , $A$ and $\omega$ such that
 \begin{align*}
    \norm{\omega_{\varphi(t_0)} -\omega}{C^{l, \frac{1}{2}}}< C_l.
 \end{align*}
\label{prn:HL27_2}
\end{proposition}

Proposition~\ref{prn:HL27_2} is the high dimension correspondence of Proposition~\ref{prn:HL27_1}.
The proof of Proposition~\ref{prn:HL27_2} is almost the same as that  of Proposition~\ref{prn:HL27_1}, except that we use Theorem~\ref{thm:HL29_2}
to improve regularity, instead of Theorem~\ref{thm:HL29_1}.

Now we are ready to prove our main theorem for the extremal K\"ahler metrics.

\begin{proof}[Proof of Theorem~\ref{thmin:1}]

Pick $\omega_{\varphi} \in \mathcal{LH}$, we need to show that the modified Calabi flow starting from $\omega_{\varphi}$ converges to $\varrho^*(\omega)$ for some $\varrho \in Aut_0(M,J)$.
For simplicity of notation, each flow mentioned in the remaining part of this proof is the modified Calabi flow.
According to the definition of $\mathcal{LH}$, we can choose a path $\varphi_s, \; s \in [0,1]$ connecting $\omega$ and $\omega_{\varphi}$, i.e.,
\begin{align*}
  \omega_{\varphi_0}=\omega, \quad \omega_{\varphi_1}=\omega_{\varphi}.
\end{align*}
Let $I$ be the collection of $s \in [0, 1]$ such that the flow starting from $\omega_{\varphi_s}$ converges.
In order to show the convergence of the flow starting from $\omega_{\varphi}$, it suffices to show the openness and closedness of $I$, since $I$ obviously contains at least one element $s=0$.

For each $s$, let $\varphi_s(t)$ be the  the time-$t$-K\"ahler-potential of the flow starting from $\omega_{\varphi_s}$.
Define function
\begin{align*}
  \hat{d}(s,t) \triangleq \inf_{\varrho \in Aut_0(M,J)}  \norm{\varrho^* \omega_{\varphi_s(t)} -\omega}{C^{k,\frac{1}{2}}}.
\end{align*}
Note that bounded closed set in $Aut_0(M,J)$ is compact since $Aut_0(M,J)$ is a finite dimensional Lie group.  Therefore, we can always find a
$\varrho_{s,t} \in Aut_0(M,J)$ such that
\begin{align*}
  \hat{d}(s,t)=  \norm{\varrho_{s,t}^* \omega_{\varphi_s(t)}-\omega}{C^{k,\frac{1}{2}}}.
\end{align*}
Therefore, $\hat{d}$ is a well defined function.  It is also clear that $\hat{d}$ depends on $s,t$ continuously.
One can refer Lemma~\ref{lma:SL07_1} and~\ref{lma:SL07_2} for a similar, but more detailed discussion.
By Theorem~\ref{thm:HL24_1}, if $\hat{d}(s_0,t_0)<\delta_0$ for some $s_0 \in [0,1]$, $t_0 \in [0, \infty)$,  then the flow starting from $\omega_{\varphi_{s_0}(t_0)}$
converges.   Consequently, the flow starting from $\omega_{\varphi_{s_0}(0)}$ converges.
By continuous dependence of the metrics on the initial data, we see that  $I$ is an open set.

We continue to show that $I$ is also closed.    Without loss of generality, we can assume that $[0,\bar{s}) \subset I$ and it suffices to show that $\bar{s} \in I$.

For each $s \in [0,\bar{s})$, let $T_s$ be the first time such that $\hat{d}(s,t)=0.5 \delta_0$.  By the convergence assumption and the continuity of $\hat{d}$, each $T_s$ is a bounded number.

\begin{claim}
The bound of $T_s$ is uniform, i.e., $\displaystyle  \sup_{s \in [0,\bar{s})} T_s < \infty.$
\label{clm:HL23_1}
\end{claim}

If the Claim fails, we can find a sequence $s_i \to \bar{s}$ such that $T_{s_i} \to \infty$.
For simplicity of notation, we denote $T_{s_i}$ by $T_i$.
Consider the flow starting from $\varphi_{\bar{s}}$.
By the results of Streets(c.f.~\cite{[Streets1]}) and He(c.f.~\cite{[He5]}), we see that
\begin{align*}
   \lim_{t \to \infty} Ca(\omega_{\varphi_{\bar{s}}}(t))=Ca(\omega).
\end{align*}
Therefore,  for each fixed $\epsilon$, we can find $L_{\epsilon}$ such that  $Ca(\omega_{\varphi_{\bar{s}}}(L_{\epsilon}))<Ca(\omega)+\epsilon$.
By continuity, we have
\begin{align}
   \lim_{i \to \infty} Ca(\omega_{\varphi_{s_i}(L_{\epsilon})})=Ca(\omega_{\varphi_{\bar{s}}}(L_{\epsilon}))<Ca(\omega)+\epsilon.
\label{eqn:HL23_3}
\end{align}
Note that the Calabi energy is non-increasing along each flow.    Since $T_i-1>L_{\epsilon}$ for large $i$,  it follows from (\ref{eqn:HL23_3}) that
\begin{align*}
     \lim_{i \to \infty} Ca(\omega_{\varphi_{s_i}(T_i-1)}) \leq Ca(\omega)+\epsilon
\end{align*}
for each fixed positive $\epsilon$.
Since the Calabi energy is always bounded from below by $Ca(\omega)$, it is clear that
\begin{align}
    \lim_{i \to \infty} \left| Ca(\omega_{\varphi_{s_i}(T_i-1)}) -Ca(\omega_{\varphi_{s_i}(T_i)}) \right| =\lim_{i \to \infty} \left( Ca(\omega_{\varphi_{s_i}(T_i-1)}) -Ca(\omega_{\varphi_{s_i}(T_i)}) \right)=0.
\label{eqn:HL23_1}
\end{align}
Note that the harmonic scale of $\omega_{\varphi_{s_i}(T_i)}$ is uniformly bounded away from zero.
By Theorem~\ref{thm:HL29_1}, we obtain all the curvature and curvature higher derivative bounds of $\omega_{\varphi_{s_i}(T_i)}$.
Define
 \begin{align*}
 \varrho_i \triangleq \varrho_{s_i,T_i},  \quad \omega_{\hat{\varphi}_i} \triangleq \varrho_i^*(\omega_{\varphi_{s_i}(T_i)}).
 \end{align*}
Note that $\omega_{\hat{\varphi}_i}$ has uniform $C^{k,\frac{1}{2}}$-norm by the choice of $T_i$.
By Proposition~\ref{prn:HL27_1}, we see that $\omega_{\hat{\varphi}_i}$ has uniform $C^{l,\frac{1}{2}}$-norm for each integer $l \geq k$.
Therefore, we can take the smooth convergence in fixed coordinate charts.
\begin{align}
  \hat{\varphi}_i \longright{C^{\infty}}   \hat{\varphi}_{\infty}, \quad \omega_{\hat{\varphi}_i}   \longright{C^{\infty}}  \omega_{\hat{\varphi}_{\infty}}.
\label{eqn:HL23_2}
\end{align}
Since the Calabi energy converges in the above process, by (\ref{eqn:HL23_1}) and the above inequality, we obtain that $\omega_{\hat{\varphi}_{\infty}}$
is an extK metric.  By the uniqueness theorem of extK metrics, we obtain that there is a $\varrho_{\infty} \in Aut_0(M, J)$ such that
\begin{align*}
    \omega_{\hat{\varphi}_{\infty}}= \varrho_{\infty}^* \omega.
\end{align*}
Note that $\varrho_{\infty}$ is automatically smooth.
Now (\ref{eqn:HL23_2}) can be rewritten as
\begin{align*}
   \varrho_i^*(\omega_{\varphi_{s_i}(T_i)})  \longright{C^{\infty}}  \varrho_{\infty}^* \omega,
   \quad \left( \varrho_i \circ \varrho_{\infty}^{-1} \right)^* (\omega_{\varphi_{s_i}(T_i)})= \left( \varrho_{\infty}^{-1}\right)^{*} \circ \varrho_i^*(\omega_{\varphi_{s_i}(T_i)}) = \left( \varrho_{\infty}^{-1}\right)^{*} \omega_{\hat{\varphi}_i} \longright{C^{\infty}} \omega.
\end{align*}
It follows from the definition of $\hat{d}$ that
\begin{align*}
 \hat{d}(s_i, T_i) \leq \norm{  \left( \varrho_i \circ \varrho_{\infty}^{-1} \right)^* (\omega_{\varphi_{s_i}(T_i)})-\omega}{C^{k,\frac{1}{2}}}  \to 0,
\end{align*}
which contradicts our choice of $T_i$, i.e., $\hat{d}(s_i, T_i)=0.5 \delta_0$.  Therefore, the proof of Claim~\ref{clm:HL23_1} is complete.

In light of Claim~\ref{clm:HL23_1}, we can choose a sequence of $s_i \to \bar{s}$ and $T_i=T_{s_i} \to \bar{T}<\infty$.  By continuity of $\hat{d}$, we see that
\begin{align*}
   \hat{d}(\bar{s}, \bar{T})=0.5 \delta_0.
\end{align*}
Consequently, Theorem~\ref{thm:HL24_1} applies and the flow starting from $\omega_{\varphi_{\bar{s}}(\bar{T})}$ converges.
Hence the flow starting from $\omega_{\varphi_{\bar{s}}(0)}$ converges and $\bar{s} \in I$.
The proof of the Theorem is complete.

\end{proof}

Theorem~\ref{thmin:3} can be proved almost verbatim,  except replacing Propsotion~\ref{prn:HL27_1} by  Proposition \ref{prn:HL27_2}.





\subsection{Convergence of complex structures}
\label{subset:complex-structures}

In this subsection, we regard the Calabi flow as the flow of the complex structures on a given symplectic manifold $(M,\omega)$.
We also assume this symplectic manifold has a cscK complex structure $J_0$.
Note that the uniqueness theorem of extK metrics in a given K\"ahler class of a fixed complex manifold plays an important role in the convergence of
potential Calabi flow.  Similar uniqueness theorem will play the same role in the convergence of complex structure Calabi flow.
Actually, by the celebrated work of Chen-Sun(c.f. Theorem 1.3 of ~\cite{[ChenSun]}), on each $C^{\infty}$-closure of
a $\mathcal{G}^{\C}$-leaf of a smooth structure $J$, there is at most one cscK complex structure(i.e., the metric determined by $\omega$ and $J$ is cscK), if it exists.

It is important to note that the complex structure Calabi flow solution is invariant under Hamilatonian diffeomorphism.
Suppose $J_A$ and $J_B$ are two isomorphic complex structures, i.e.
\begin{align*}
   \varphi^* \omega = \omega, \quad J_A=\varphi^* J_B
\end{align*}
for some symplectic diffeomorphism $\varphi$.  Let $J_A(t)$ be a Calabi flow solution starting from $J_A$, then $\varphi^* J_A(t)$ is a Calabi flow
solution starting from $J_B$.

Let $g_0$ be the metric compatible with $\omega$ and $J_0$. Then it
is clear that $g_0$ is cscK and therefore smooth metric. We can
choose coordinate system of $M$ such that $g,J,\omega$ are all
smooth in each coordinate chart.  We equip the
tangent bundle and cotangent bundle and their tensor products with
the natural metrics induced from $g_0$. Clearly, if a diffeomorphism
$\varphi$ preserves both $\omega$ and $J_0$, then it preserves $g_0$
and therefore locates in $\mathrm{ISO}(M, g_0)$, which is compact
Lie group. Therefore, each $\varphi$ is smooth and has a priori bound
of $C^{l,\frac{1}{2}}$-norm for each positive integer $l$, whenever
$\varphi$ is regarded as a smooth section of the bundle $M \times M \to M$,
equipped with natural  metric induced from $g_0$.
There is an almost version of this property. In other words, if $\varphi$ preserves $\omega$ and the $C^{k,\frac{1}{2}}$-norm of $\varphi^* J$
is very close to $J_0$, then $\varphi$ has a priori bound of $C^{k+1,\frac{1}{2}}$-norm. This is basically because of the
improving regularity property of isometry(c.f.~\cite{[CH]}).

\begin{proposition}
  Let $\left\{U_x, \left\{ x^i \right\}_{i=1}^{m} \right\}$ and $\left\{U_y, \left\{ y^i \right\}_{i=1}^{m} \right\}$
 be two coordinates of an open Riemannian manifold $(V, ds^2)$.
 The Riemannian metric in these two coordinates can be written as
  \begin{align}
    ds^2=g_{ij}(x) dx^i dx^j=\tilde{g}_{ij}(y) dy^i dy^j.
    \label{eqn:HL07_2}
  \end{align}
 As subsets of $\R^m$, $U_x$ and $U_y$ satisfy
 \begin{align*}
    &\left\{ x \,\Big|\,|x|=\sqrt{x_1^2 + x_2^2 + \cdots + x_m^2}<\frac{1}{2}  \right\} \subset U_x \subset \left\{x\,|\,|x|<1 \right\}, \\
    &\left\{y \,\Big|\,|y|<\frac{1}{2}  \right\} \subset U_y \subset \left\{y\,|\, |y|<1 \right\}.
  \end{align*}
 In each coordinate, the metrics are uniform equivalent to Euclidean metrics.
 \begin{align*}
   & \frac{1}{2} \delta_{ij} < g_{ij}(x) < 2 \delta_{ij},  \quad \forall \; x \in U_x \\
   & \frac{1}{2} \delta_{ij} < \tilde{g}_{ij}(y) < 2 \delta_{ij}, \quad \forall \; y \in U_y.
 \end{align*}
 Suppose $g_{ij}$ and $\tilde{g}_{ij}$ are of class $C^{k,\frac{1}{2}}$ for some integer $1 \leq k \leq \infty$.
 Suppose the natural map $x=f(y)$ satisfies $f(0)=0$. Then $f$ is of class $C^{k+1, \frac{1}{2}}$ and
 \begin{align*}
   \norm{f}{C^{k+1,\frac{1}{2}}(B_{\frac{1}{4}})} < C \left(m,k, \norm{g_{ij}}{C^{k,\frac{1}{2}}(U_x)} ,  \norm{\tilde{g}_{ij}}{C^{k,\frac{1}{2}}(U_y)}   \right),
 \end{align*}
 where $B_{\frac{1}{4}}$ is the standard ball in $\R^m$ with radius $\frac{1}{4}$ and centered at $0$.
\label{prn:HL07_1}
\end{proposition}

\begin{proof}
  Equation (\ref{eqn:HL07_2}) can be rewritten as
\begin{align*}
  \tilde{g}_{ij}=g_{kl} \D{x^k}{y^i} \D{x^l}{y^j}.
\end{align*}
Denote $\Gamma_{ij}^k$ and $\tilde{\Gamma}_{ij}^k$ as the Christoffel symbol of $ds^2$ under the $x$ and $y$ coordinate respectively.
Then direct calculation implies that
\begin{align*}
  \tilde{\Gamma}_{ij}^k=\D{x^p}{y^i} \D{x^q}{y^j} \Gamma_{pq}^r \D{y^k}{x^r} + \D{y^k}{x^u} \frac{\partial^2 x^u}{\partial y^i \partial y^j}.
\end{align*}
In other words, the second derivatives of $x(y)$ can be expressed as
\begin{align}
  \frac{\partial^2 x^u}{\partial y^i \partial y^j}=\frac{\partial x^u}{\partial y^k} \tilde{\Gamma}_{ij}^k
  -\Gamma_{pq}^u \D{x^p}{y^i} \D{x^q}{y^j}.
  \label{eqn:HL07_3}
\end{align}
Now suppose the metric $\tilde{g}_{ij}$ is $C^{k,\frac{1}{2}}$ for some integer $k \geq 1$.
Moreover, let us assume
\begin{align*}
  \norm{g_{ij}}{C^{k,\frac{1}{2}}(U_x)},  \quad \norm{\tilde{g}_{ij}}{C^{k,\frac{1}{2}}(U_y)}  \leq  A.
\end{align*}
Therefore, $\tilde{\Gamma}_{ij}^k$ and $\Gamma_{ij}^k$ are $C^{k-1,\frac{1}{2}}$ and has uniformly bounded
$C^{k-1, \frac{1}{2}}$-norm in smaller balls. By bootstrapping argument, we obtain that
\begin{align*}
  \norm{f}{C^{k+1,\frac{1}{2}}(B_{\frac{1}{4}})} \leq C(m,k,A).
\end{align*}
\end{proof}

\begin{lemma}
  Suppose $\varphi \in Symp(M,\omega)$.  Then we have
  \begin{align*}
    \norm{\varphi}{C^{k+1, \frac{1}{2}}} < C \left( n,k,\norm{\varphi_* J}{C^{k,\frac{1}{2}}}, \norm{J}{C^{k,\frac{1}{2}}}  \right).
  \end{align*}
  \label{lma:HL07_3}
\end{lemma}

\begin{proof}
  Regard $\varphi$ as an isometry from $(M,\omega,J)$ to $(M,\omega, \varphi_{*} J)$.
  Then the proof boils down to Proposition~\ref{prn:HL07_1}.
  Note that $\varphi_* J$ is the push-forward of $J$, which is the same as $(\varphi^{-1})^* J$.
\end{proof}

Fix $J_0$ as the cscK complex structure.
For each complex structure $J$ compatible with $\omega$, we define
   \begin{align*}
      \tilde{d}(J)=\inf \norm{\varphi^*  J- J_0}{C^{k,\frac{1}{2}}},
   \end{align*}
where infimum is taken among all symplectic diffeomorphisms with
finite $C^{k,\frac{1}{2}}$-norm. Let $\varphi_i$ be a minimizing
sequence to approximate $\tilde{d}(J)$. By triangle inequality and
Lemma~\ref{lma:HL07_3}, we see that $\varphi_i$ has uniformly
bounded $C^{k+1,\frac{1}{2}}$-norm.
Therefore, by taking subsequence if necessary, we can assume
$\varphi_i$ converges, in the $C^{k+1,\frac{1}{3}}$-topology, to a
limit symplectic diffeomorphism $\varphi_{\infty}$. Although the
convergence topology is weak, it follows from definition that
$\varphi_{\infty}$ has bounded $C^{k+1,\frac{1}{2}}$-norm. Therefore,
we see that
\begin{align*}
  \norm{\varphi_{\infty}^* J-J_0}{C^{k,\frac{1}{2}}} \leq \lim_{i \to \infty}  \norm{\varphi_i^* J-J_0}{C^{k,\frac{1}{2}}}
  =\tilde{d}(J).
\end{align*}
On the other hand, by definition, we have
\begin{align*}
   \tilde{d}(J) \leq \norm{\varphi_{\infty}^* J-J_0}{C^{k,\frac{1}{2}}}.
\end{align*}
Combining the above two inequalities, we obtain that $\tilde{d}(J)$ is achieved by $\varphi_{\infty}^* J$.
We have proved the following Lemma.

 \begin{lemma}
 For each smooth complex structure $J$ compatible with $\omega$,
  $\tilde{d}(J)$ is achieved by a diffeomorphism $\varphi \in C^{k+1,\frac{1}{2}}(M,M)$ and $\varphi^* \omega=\omega$.
   \label{lma:SL07_1}
 \end{lemma}

 \begin{lemma}
   $\tilde{d}$ is a continuous function on the moduli space of complex structures, equipped with $C^{k,\frac{1}{2}}$-topology.
   \label{lma:SL07_2}
 \end{lemma}

 \begin{proof}
    Fix $J_A$, let $J_B \to J_A$. Then we need to show that
    \begin{align*}
      \tilde{d}(J_A)=\lim_{J_B \to J_A} \tilde{d}(J_B).
    \end{align*}
    On one hand, by Lemma~\ref{lma:SL07_1}, we can find $\varphi_A \in C^{k+1,\frac{1}{2}}(M,M)$ such that
    \begin{align*}
      \tilde{d}(J_A)=\norm{\varphi_A^* J_A-J_0}{C^{k,\frac{1}{2}}}.
    \end{align*}
    It follows that
    \begin{align*}
      \tilde{d}(J_B)=\inf_{\varphi \in Symp(M,\omega)} \norm{\varphi^* J_B-J_0}{C^{k,\frac{1}{2}}} \leq \norm{\varphi_A^* J_B-J_0}{C^{k,\frac{1}{2}}}.
    \end{align*}
    Let $J_B \to J_A$ and take limit on both sides, we have
    \begin{align*}
      \limsup_{J_B \to J_A} \tilde{d}(J_B) \leq \limsup_{J_B \to J_A} \norm{\varphi_A^* J_B-J_0}{C^{k,\frac{1}{2}}} =\norm{\varphi_A^* J_A-J_0}{C^{k,\frac{1}{2}}}= \tilde{d}(J_A).
    \end{align*}
    On the other hand, for each $J_B$, we have $\varphi_{B}$ to achieve the $\tilde{d}$. Then we see that
    \begin{align*}
      \liminf_{J_B \to J_A} \tilde{d}(J_B)=\liminf_{J_B \to J_A} \norm{\varphi_B^* J_B-J_0}{C^{k,\frac{1}{2}}} \geq
      \norm{\varphi_{\infty}^* J_A-J_0}{C^{k,\frac{1}{2}}}\geq \tilde{d}(J_A),
    \end{align*}
    where $\varphi_{\infty}$ is a limit symplectic diffeomorphism of $\varphi_{B}$ as $J_B \to J_A$. The existence and estimates of $\varphi_{\infty}$
    follow from Lemma~\ref{lma:HL07_3}.

    Therefore, $\tilde{d}$ is continuous at $J_A$ by combining the above two inequalities.   Since $J_A$ is chosen arbitrarilly in the moduli space of complex structures, we
    finish the proof.
 \end{proof}

\begin{lemma}
Suppose $J_A=J_s(t)$ for some $s \in D$ and $t \geq 0$.
If $\tilde{d}(J_A)<\delta$, then the Calabi flow starting from $J_A$ has global existence and converges to  $\psi^*(J_0)$ for some
 $\psi \in \mathrm{Symp}(M, \omega)$.
\label{lma:SL05_1}
\end{lemma}

\begin{proof}
   By Theorem 5.3 of \cite{[ChenSun]},
  there is a small $\delta$ such that every Calabi flow starting from the $\delta$-neighborhood of $J_0$, in $C^{k,\frac{1}{2}}$-topology, will converge to some cscK $J'$.
  Note that the Calabi flow solution always stays in the $\mathcal{G}^{\C}$-leaf of $J_1$, hence $J'$ is in the $C^{\infty}$-closure of $J_1$.
  Also, on the other hand,  according to the conditions of $J_s$, we see that $J_0$ is also in the $C^{\infty}$-closure of $J_1$.
  Therefore,  one can apply the uniqueness degeneration theorem, Theorem 1.3 of~\cite{[ChenSun]} to obtain that $(M, \omega, J_0)$
  and $(M, \omega, J')$ are isomorphic.
  Namely, $(\omega, J')=\eta^*(\omega, J_0)$ for some diffeomorphism $\eta$.

  If $\tilde{d}(J_A)<\delta$, then we can find a diffeomorphism $\varphi$ such that
  \begin{align*}
    \varphi^* \omega=\omega,  \quad \norm{\varphi^*J_A-J_0}{C^{k,\frac{1}{2}}(M,g_0)}<\delta.
  \end{align*}
  By previous argument, the Calabi flow initiated from $\varphi^* J_A$ will converge to $\eta^*(J_0)$, for some symplectic diffeomorphism $\eta$.
  Then the Calabi flow starting from $J_A$ converges to $(\varphi^{-1})^* \eta^* (J_0)$.
  Let $\psi=\eta \circ \varphi^{-1}$, we then finish the proof.
\end{proof}

With these preparation, we are ready to prove Theorem~\ref{thmin:2}.

\begin{proof}[Proof of Theorem~\ref{thmin:2}]

     Let $I$ to be the collection of $s \in [0,1]$ such that
  the Calabi flow initiated from $J_s$ converges to $\psi^*(J_0)$ for some symplectic diffeomorphism $\psi$.
  By abusing of notation, we denote
  \begin{align*}
     \tilde{d}(s,t)=\tilde{d}(J_s(t)).
  \end{align*}
  In light of Lemma~\ref{lma:SL05_1} and continuity of $\tilde{d}$, we see that $I$ is open.  In order to show $I=[0,1]$,  it suffices to show the following claim.

  \begin{claim}
   Suppose $[0,\bar{s}) \subset I$ for some $\bar{s} \in (0,1]$, then $\bar{s} \in I$.
  \label{clm:SL05_1}
  \end{claim}

  We argue by contradiction.

  If the statement was wrong, then $\tilde{d}(\bar{s}, 1) \geq \delta$, due to Lemma~\ref{lma:SL05_1}.
  For each $s \in (0,\bar{s})$, we see $\tilde{d}(s, t)$ will converge to zero finally, while $\tilde{d}(s,1)>0.5 \delta$ whenever
  $s$ is very close to $\bar{s}$, say, for $s \in [\bar{s}-\xi, \bar{s})$.
  Therefore, for each $s$ nearby $\bar{s}$, there is a $T_{s}$ such that $\tilde{d}(s, T_s)=0.5 \delta$ for the first time.
  A priori, there are two possibilities for the behavior of $T_{s}$:

  \begin{itemize}
   \item  There is a sequence of $s_i$ such that $s_i \to \bar{s}$ and $T_{s_i} \to \infty$.
   \item  $\displaystyle \sup_{s \in [\bar{s}-\xi,\bar{s})}T_s <A$ for some constant $A$.
  \end{itemize}
  We shall exclude the first possibility. The proof is parallel to that of Claim~\ref{clm:HL23_1}.
  Actually, if the first possibility appears, then we see that
  \begin{align*}
    \lim_{i \to \infty} Ca(J_{s_i}(T_{s_i}))=0.
  \end{align*}
  In light of the monotonicity of the Calabi energy along each flow,  the continuous dependence of the the Calabi energy on parameters $s$ and $t$, and the fact that
  \begin{align*}
     \lim_{t \to \infty}  Ca(J_{\bar{s}}(t))=0.
  \end{align*}
  For simplicity of notations, we denote $J_i'=J_{s_i}( T_{s_i})$.
  Note that  every $J_i'$ can be pulled-back to $J_1'$ by some Hamiltonian diffeomorphisms $\eta_i$, which may not preserve $\omega$.
  In other words, we have
    \begin{align}
        J_1'=\eta_i^{*} J_i', \quad \omega_i \triangleq \eta_i^{*} \omega,
    \label{eqn:HL15_0}
    \end{align}
   where $\omega_i=\omega + \sqrt{-1} \partial \bar{\partial} f_i$ for some smooth functions $f_i$, with respect to the complex structure $J_1'$.
   Note also that $(M, \omega, J_i')$ has the same intrinsic geometry as $(M, \omega_i, J_1')$, which has uniformly bounded Riemannian geometry and all high order curvature covariant derivatives, due to Theorem~\ref{thm:HL29_1}.
   Therefore, we can take Cheeger-Gromov convergence:
    \begin{align*}
       (M, \omega_i, J_1') \longright{Cheeger-Gromov-C^{\infty}}   \left(M, \bar{\omega}, \bar{J} \right).
    \end{align*}
    In other words, there exist smooth diffeomorphisms $\varphi_i$ such that
    \begin{align}
      \varphi_i^* \omega_i \stackrel{C^{\infty}}{\longrightarrow}  \bar{\omega}, \quad \varphi_i^* J_1' \stackrel{C^{\infty}}{\longrightarrow}  \bar{J}.
    \label{eqn:HL15_1}
    \end{align}
    Note that everything converges smoothly, the limit K\"ahler manifold $(M, \bar{\omega}, \bar{J})$
    has zero Calabi energy and consequently is cscK. Moreover,  it is adjacent to $([\omega], J_1')$, in the sense of Chen-Sun(c.f. Definition 1.4 of~\cite{[ChenSun]}).
    Therefore, the unique degeneration theorem of Chen-Sun (Theorem 1.6 of~\cite{[ChenSun]}) applies, we see that there is a smooth diffeomorphism $\varphi$ of $M$ such that
    \begin{align}
      \bar{\omega}=\varphi^* \omega, \quad \bar{J}=\varphi^{*} J_0.
    \label{eqn:HL15_2}
    \end{align}
    Let $\psi_i =\varphi_i \circ \varphi^{-1}$, $\gamma_i=\eta_i \circ \psi_i$. Combining (\ref{eqn:HL15_0}), (\ref{eqn:HL15_1}) and (\ref{eqn:HL15_2}),  we obtain
    \begin{align*}
       &\psi_i^* \omega_i \stackrel{C^{\infty}}{\longrightarrow} \omega,  \quad     \psi_i^* J_1' \stackrel{C^{\infty}}{\longrightarrow} J_0, \\
       & \gamma_i^* \omega \stackrel{C^{\infty}}{\longrightarrow} \omega, \quad   \gamma_i^* J_i'  \stackrel{C^{\infty}}{\longrightarrow}  J_0.
    \end{align*}
    Since $[\omega]$ is integral, we see $[\gamma_i^* \omega]=[\omega]$ for sufficiently large $i$.
    Composing with an extra convergent sequence of diffeomorphisms if necessary, we can assume that
    \begin{align*}
       \gamma_i^* \omega  =\omega,  \quad     \gamma_i^* J_i' \stackrel{C^{\infty}}{\longrightarrow} J_0.
    \end{align*}
    Then we have $\tilde{d}(J_i') \to 0$ as $i \to \infty$, which contradicts our choice of $J_i'$, namely, $\tilde{d}(J_i')=0.5 \delta$.
    This contradiction excludes the first possibility.
    Therefore, we have
    \begin{align*}
       \sup_{s \in [\bar{s}-\xi,\bar{s})}T_s <A
    \end{align*}
    for some uniform constant $A$.  Now we choose $s_i \in [0, \bar{s})$ such that $s_i \to \bar{s}$, we can assume $T_{s_i} \to T_{\bar{s}}$ for some finite $T_{\bar{s}}$
    by the above estimate.   In light of the continuous dependence of solutions to the initial data, we see that the complex structure $J_{\bar{s}}(T_{\bar{s}})$
      is the limit of $J_{s_i}(T_{s_i})$.  Hence
      \begin{align*}
         \tilde{d}(J_{\bar{s}}(T_{\bar{s}}))= \lim_{i \to \infty} \tilde{d}(J_{s_i}(T_{s_i}))=0.5 \delta < \delta.
      \end{align*}
    Consequently, Lemma~\ref{lma:SL05_1} can be applied for the Calabi flow started from $J_{\bar{s}}(T_{\bar{s}})$. Therefore, $\bar{s} \in I$ and we finish the proof
    of Claim~\ref{clm:SL05_1}.
\end{proof}

The proof of Theorem~\ref{thmin:4} is almost the same as Theorem~\ref{thmin:2}.
The only difference is that we use Theorem~\ref{thm:HL29_2} to improve regularity, rather than Theorem~\ref{thm:HL29_1}.


\section{Further discussion}

\subsection{Examples with global existence}

All the convergence results in Section~\ref{sec:longtime-behavior} are based on global existence of the Calabi flows.
In this subsection, we will show some non-trivial examples with global existence.
On such examples,  our results and methods developed in Section~\ref{sec:longtime-behavior} can be applied.

\begin{example}
\label{example1}(cf. \cite{[ChenHe2]})
Let $(M,J)$ be the blowup of $\CC\PP^2$ at three generic points. Let the K\"ahler class of $\omega$ be
\begin{align*}
  3[H]-\la ([E_1]+[E_2]+[E_3]),  \quad 0<\lambda<\frac{3}{2},   \quad \lambda \in \Q,
\end{align*}
where $H$ denotes the pullback of the hyperplane of $\CC\PP^2$ in $M$ and $E_i(1\leq i\leq 3)$ denotes the exceptional divisors.
Suppose $\oo$ is invariant under the toric action and the action of $\,\ZZ_3$ and satisfies
\beq\label{example1energy}
\int_M\; S^2\,dV<192\pi^2+32\pi^2\frac {(3-\la)^2}{3-\la^2}.
\eeq
Then the Calabi flow starting from $\oo$ exists for all time and converges to a cscK metric in the smooth topology of the K\"ahler potentials.
\end{example}

\begin{example}
\label{example2}(cf. \cite{[ChenHe3]})
Let $(M,J)$ be a toric Fano surface. Let $\omega$ be a K\"ahler metric with positive extremal Hamiltonian potential.
Suppose $[\omega]$ is rational and $\oo$ is invariant under the toric action and satisfies
\begin{align*}
\int_M\; S^2 \,dV<32\pi^2 \left(c_1^2(M)+\frac 13 \frac {(c_1(M)\cdot \Omega)^2}{\Omega^2}\right)+\frac 13 \|\cF\|^2,
\end{align*}
where $\|\cF\|^2$ is the norm of Calabi-Futaki invariant.  Then the modified Calabi flow starting from $\oo$ exists for all time and converges
to an extK metric in the smooth topology of K\"ahler potentials.
\end{example}

Actually, in~\cite{[ChenHe2]} and~\cite{[ChenHe3]}, the global existence of the flows in Example \ref{example1} and \ref{example2} was already proved by energy method,
based on the work of~\cite{MR2425183}.  Furthermore, they showed the sequence convergence in the Cheeger-Gromov topology.
The only new thing here is the improvement of the convergence topology.
Let us sketch a proof of the statement of Example~\ref{example2}.    By results of~\cite{[ChenHe3]}, we can take time sequence $t_i \to \infty$ such that
\begin{align*}
    (M, \omega(t_i), J)  \longright{Cheeger-Gromov-C^{\infty}} (M', \omega', J'),
\end{align*}
where $(M', \omega', J')$ is an extK metric.   In other words, there exist smooth diffeomorphisms $\psi_i$ such that
\begin{align*}
  (\psi_i^* \omega(t_i), \psi_i^*J)   \longright{C^{\infty}}   (\omega', J')
\end{align*}
in fixed coordinates.   The toric symmetry condition forces that $J'$ is isomorphic to $J$, i.e., $J'=\psi^* J$ for some smooth diffeomorphism $\psi$.  Let $\eta_i=\psi_i \circ \psi^{-1}$, we have
\begin{align}
  (\eta_i^* \omega(t_i), \eta_i^* J) \longright{C^{\infty}}   ((\psi^{-1})^* \omega', J).
\label{eqn:HL29_3}
\end{align}
The rational condition of $[\omega]$ forces that $[\eta_i^* \omega(t_i)]=[(\psi^{-1})^* \omega']$. Therefore, there exists smooth diffeomorphisms $\rho_i \to Id$ such that
\begin{align}
 \rho_i^*\eta_i^* \omega(t_i)=(\psi^{-1})^* \omega'.   \label{eqn:HL29_4}
\end{align}
Then equation (\ref{eqn:HL29_3}) becomes
\begin{align}
 \rho_i^* \eta_i^* J \longright{C^{\infty}}   J.    \label{eqn:HL29_5}
\end{align}
We can write $\eta_i= \varrho_i \circ \sigma_i$ for $\varrho_i \in Aut(M,J)$ and $\sigma_i \to Id$.  Then (\ref{eqn:HL29_3}) can be rewritten as
\begin{align*}
   & (\sigma_i^* \varrho_i^* \omega(t_i), \sigma_i^* J) \longright{C^{\infty}}   ((\psi^{-1})^* \omega', J), \\
   & ( \varrho_i^* \omega(t_i), J) \longright{C^{\infty}}   ((\psi^{-1})^* \omega', J).
\end{align*}
Note that $Aut(M,J)=Aut_0(M,J)$ in our examples.  Hence
$[\varrho_i^* \omega(t_i)]=[(\psi^{-1})^* \omega']=[\omega]$.  For simplicity of notation, we denote $(\psi^{-1})^* \omega'$ by $\omega_{extK}$.  Then on the fixed complex manifold $(M, J)$, within
the K\"ahler class $[\omega]$, we have $\varrho_i^* \omega(t_i) \to \omega_{extK}$ in the smooth topology of K\"ahler potentials.
For some large $i$,  $\varrho_i^* \omega(t_i)$ locates in a tiny $C^{k,\frac{1}{2}}$-neighborhood of $\omega_{extK}$.
Then we can apply Theorem~\ref{thm:HL24_1} and Theorem~\ref{thm:HL24_exp} to show the modified Calabi flow staring from $\varrho_i^* \omega(t_i)$ converges to
an extK metric $\varrho^* \omega_{extK}$ exponentially fast, where $\varrho \in Aut_0(M,J)$.

\subsection{Behavior of the Calabi flow at possible finite singularities}
\label{subsec:finite}

All our previous discussion in this paper is based on the global existence of the Calabi flow, and there do exist some nontrivial examples of global existent Calabi flow.
However, it is not clear whether the global existence of the Calabi flow holds in general.
Suppose the Calabi flow starting from $\omega_{\varphi}$ fails to have global existence. Then there must be a maximal existence time $T$.
By the work of Chen-He~\cite{[ChenHe1]}, we see that Ricci curvature must blowup at time $T$.     In this subsection, we will study the behavior of more geometric quantities
at the first singular time $T$.  For simplicity of notations, we often let the singular time $T$ to be $0$.
Recall that $P,Q,R$ are defined in equation (\ref{eqn:SL06_1}).

\begin{proposition}
 Suppose $\{(M^n, g(t)), -1 \leq t \leq K, 0 \leq K\}$ is a Calabi flow solution satisfying
$$Q(0)=1, \quad Q(t)\leq 2,\quad \forall\; t\in [-1, 0].$$
Then we have
\begin{align}
Q(K)<2^{\frac {1}{\ee_0}\int_0^K\,P_g(t)\;dt+1},
\label{eqn:HG24_1}
\end{align}
where $\ee_0$ is the dimensional constant obtained in Lemma~\ref{lem002}.
\label{prn:HG24_1}

\end{proposition}

\begin{proof} For any nonnegative integer $i$,
 we define $s_i=\inf\left\{t\;|\;t\geq 0, Q(t)=2^i \right\}$.
 Note that  $s_0=0$. Thus, we have
$$s_i-\frac {1}{Q(s_i)^2}\geq -1,\quad \sup_{\left[s_i-\frac 1{Q(s_i)^2}, \;s_{i+1} \right]}Q(t)=
Q(s_{i+1})=2^{i+1}.$$
We rescale the metrics by
$$g_i(x, t) \triangleq Q(s_i)g \left(x, \frac t{Q(s_i)^2}+s_i \right).$$
Then the flow $\left\{(M, g_i(t)),  -1 \leq t \leq  Q(s_i)^2(K-s_i) \right\}$ is a
 Calabi flow solution satisfying
 \begin{itemize}
 \item $Q_{g_i}(0)=1$,
 \item $Q_{g_i}(t)\leq 1$ for all $t\in [-1, 0]$,
 \item $Q_{g_i}\Big(Q(s_i)^2(s_{i+1}-s_i)\Big)=2$.
 \end{itemize}
Thus, Lemma \ref{lem002} applies and we have
$$\int_{s_i}^{s_{i+1}}\,P_g(t)\,dt=\int_0^{Q(s_i)^2(s_{i+1}-s_i)}\,P_{g_i}(t)\,dt\geq \ee_0.$$
Let $N$ be the largest $i$ such that $s_i\leq K$. Then
$$N\ee_0\leq \int_0^{s_N}\,P_g(t)\,dt\leq \int_0^K\,P_g(t)\,dt, $$
which implies that $N\leq \frac 1{\ee_0}\int_0^K\,P_g(t)\,dt$. It follows that for any
$t\in [0, K]$,
$$Q(t)\leq 2^{N+1}\leq 2^{\frac 1{\ee_0}\int_0^K\,P_g(t)\,dt+1}.$$
Then (\ref{eqn:HG24_1}) follows trivailly and we finish the proof.
\end{proof}

Since the Riemannian curvature tensor blows up at the singular time along the
Calabi flow,   Proposition~\ref{prn:HG24_1} directly implies the following result.

\begin{corollary}
\label{cor:001}

Suppose $\{(M^n, g(t)), -1 \leq t <0\}$ is a Calabi flow solution.
If $t=0$ is the singular time, then we have
$$\int_{-1}^0\,P_g(t)\,dt=\infty.$$
In particular,  $P_g(t)$ will blow up at the singular time $t=0$:
$$\lim_{t\ri 0} P_g(t)=\infty.$$

\end{corollary}

Next, we would like to estimate  $Q(t)$
near the singular time of the Calabi flow. Analogous results for
Ricci flow are proved by the maximum principle (cf. for example, Lemma 8.7 of \cite{[Ben]}).
Here we  show similar results for the Calabi flow by using the higher order curvature estimates.

\begin{lemma}\label{lem003}
There exists a constant $\dd_0=\dd_0(n)>0$ with the following properties.

Suppose $\{(M^n, g(t)), -1 \leq t \leq 0\}$ is a Calabi flow solution
and $t=0$ is the singular time.  Then
\beq \limsup_{t\ri 0}Q(t)\sqrt{-t}\geq \dd_0. \label{eq201}\eeq

\end{lemma}
\begin{proof}By Theorem~\ref{thm:GA04_JS}, we can find a constant
$\dd_1(n)>1$ satisfying the following
properties. If $\{(M, g(t)), -1 \leq t \leq -\frac12\}$ is a Calabi flow solution
with $|\Rm|(t)\leq 1$ for $t\in [-1, -\frac 12]$, then
\beq
 \left| \pd{}t|\Rm| \right|_{t=-\frac 12}\leq \dd_1. \label{eq609}
\eeq
We claim that under the assumption of Lemma~\ref{lem003}, we have
\beq
\sup_{[-1, -\frac 1{2\dd_1}]}Q(t)\geq \frac 12.\label{ee200}
\eeq
Otherwise, we have
$
\displaystyle \sup_{[-1, -\frac 1{2\dd_1}]}Q(t)<\frac 12.
$
We choose $t_0\in [-1, 0)$ the first time such that $Q(t_0)=1$. Clearly,
 $t_0>-\frac 1{2\dd_1}.$  Since
$$Q(t_0) \leq Q\left(-\frac 1{2\dd_1} \right)+\sup_{M\times[ -\frac 1{2\dd_1}, t_0]}
\left|\pd {}t|\Rm|\right| \cdot \left(t_0+\frac 1{2\dd_1} \right),$$
we have
\beq 1< \frac 12+\dd_1 \left(t_0+\frac 1{2\dd_1} \right) \leq \frac 12+\dd_1\cdot \frac 1{2\dd_1}\leq 1, \label{eq:100}\eeq
which is a contradiction. Here we used (\ref{eq609}) in the first inequality of (\ref{eq:100}). The inequality
(\ref{ee200}) is proved.

Now we estimate $Q(s_0)$ for any  $s_0\in (-1, 0).$ Since $t=0$ is the singular time,
we can assume that $Q(t)\leq Q(s_0)$ for all $t\in [-1, s_0]$.  Rescale the
metric by
$$\td g(x, t)=Q \,g\left(x, \frac t{Q^2} \right),\quad t\in [-Q^2,  0).$$
Choose $Q$ such that $s_0Q^2=-\frac 1{2\dd_1}$. If $|s_0|\leq \frac 1{2\dd_1}$, then we have
$Q\geq 1$ and $\{(M, \tilde{g}(t)), -1 \leq t <0\}$ is a Calabi flow solution with the singular time $t=0.$ Thus, by
(\ref{ee200}) we have
$$Q_{\td g} \left(-\frac 1{2\dd_1} \right) \geq \frac 12,$$
which is equivalent to say
$$Q(s_0)\sqrt{-s_0}\geq \frac {1}{2\sqrt{2\dd_1}},\quad \forall\; s_0\in  \left[ -\frac 1{2\dd_1}, 0 \right).$$
The lemma is proved.
\end{proof}

The next result gives an upper bound of $Q$ near the singular time with the assumption on $P$.

\begin{lemma}\label{lem004}
Suppose that $\{(M^n, g(t)), -1 \leq t <0\}$ is a Calabi flow with singular time $t=0$.
If
\beq \limsup_{t\ri
0}P(t)|t|=C<+\infty, \label{eq:C001}
\eeq
then we have \beq Q(t)=o\left(|t|^{-\la}\right) \eeq for any constant $\la>\frac {C\log 2}{\ee_0}.$
Here $\ee_0$ is the
constant in Lemma \ref{lem002}.
\end{lemma}

\begin{proof}
Since $t=0$ is the singular time,  for any $\dd>0$
we can choose $t_0=t_0(g, \dd)$ such that the following properties hold:
\begin{align}
&P(t)|t|<C+\dd,\quad \forall \; t\in [t_0, 0), \label{eq:C002}\\
&Q(t) \leq  Q(t_0), \quad \forall\;t\in [-1, t_0], \label{eq:C003}\\
&Q^2(t_0)|1+t_0| \geq 1.\label{eq:C004}
\end{align}
Define
 $s_i=\inf \left\{t\;|\; t\geq t_0,\; Q(t)=2^i Q(t_0) \right\}$ and we rescale the metrics by
 $$h_i(x, t)=Q(s_i)g \left(x, \frac {t}{Q^2(s_i)}+s_i \right),\quad t\in \Big[-(1+s_i)Q^2(s_i),
Q^2(s_i)|s_i|\Big).$$ Then by (\ref{eq:C003}) and (\ref{eq:C004}) the flow $\{(M, h_i(t)), -1 \leq t \leq 0\}$ satisfies
$$Q_{h_i}(0)=1,\quad Q_{h_i}(t)\leq 1, \quad \forall\; t\in [-1, 0],$$
and
 $Q_{h_i}\Big(Q^2(s_i)|s_{i+1}-s_i|\Big)=2. $
By Lemma \ref{lem002}, we have
\beq \int_{s_i}^{s_{i+1}}\,P\,dt=\int_0^{Q^2(s_i)|s_{i+1}-s_i|}\,
P_{h_i}(t)\,dt\geq \ee_0. \label{eq600}\eeq
On the other hand, by (\ref{eq:C002}) we have
\beq \int_{s_i}^{s_{i+1}}\,P\,dt\leq (C+\dd)\log \frac
{|s_i|}{|s_{i+1}|}. \label{eq601}\eeq
Combining the inequalities (\ref{eq600}) with (\ref{eq601}), we have
$$\frac {|s_{i+1}|}{|s_i|}\leq e^{-\frac {\ee_0}{C+\dd}}.$$
After iteration, we get the inequality
\beq |s_i|\leq |s_0|e^{-\frac {\ee_0 i}{C+\dd}}=|t_0|e^{-\frac {\ee_0
i}{C+\dd}}. \label{eq:101}\eeq
Combining (\ref{eq:101}) with the definition of $s_i$ gives the result
$$\lim_{i\ri +\infty}Q(s_i)|s_i|^{\la}\leq \lim_{i\ri
+\infty}Q(t_0)|t_0|^{\la}
\Big(2e^{-\frac {\la\ee_0}{C+\dd}}\Big)^i=0,$$
where we choose $\la$ such that $2e^{-\frac {\la\ee_0}{C+\dd}}<1.$
Thus, for any $t\in [s_i, s_{i+1}]$ we have
$$Q(t)|t|^{\la}\leq Q(s_{i+1})|s_i|^{\la}=2Q(s_{i})|s_i|^{\la}\ri 0$$
as $i\ri +\infty.$ The lemma is proved.
\end{proof}

We now prove Theorem~\ref{thmin:5}.

\begin{proof}[Proof of Theorem~\ref{thmin:5}]
It is clear that (\ref{eqn:HG22_3}) follows from the combination of inequality (\ref{eqn:HG22_2}) and the definition of type-I singularity(c.f.~\cite{[LZ]}), i.e., $\displaystyle \limsup_{t \to 0} Q^2|t|<\infty$.
Therefore, we only need to show (\ref{eqn:HG22_1}) and (\ref{eqn:HG22_2}), which will be dealt with
separately.

\textit{1. Proof of inequality (\ref{eqn:HG22_1}):}

If $\displaystyle \limsup_{t\ri 0}P|t|=C\geq 0$, by Lemma
\ref{lem003} and Lemma \ref{lem004} we have
$$0=\limsup_{t\ri 0}Q(t)|t|^{\frac {(C+\dd)\log 2}{\ee_0}}
\geq \dd_1|t|^{-\frac 12+\frac {(C+\dd)\log 2}{\ee_0}},$$ where
$\dd>0$ is any small constant. It follows that
$$C\geq \frac {\ee_0}{2\log 2}.$$
Thus, (\ref{eqn:HG22_1}) is proved.

\textit{2. Proof of inequality (\ref{eqn:HG22_2}):}

    Since $0$ is the singular time,  we have
   $\displaystyle  \lim_{t \to 0} F(t)=0$.   Thus, $F(t)\in (0, 1)$ when $t$ is close to $0$.  It follows  that
    \begin{align}
     -F(t)= \lim_{s\ri 0}F(s)-F(t) \geq -C|t|,  \Rightarrow F(t) \leq C|t|,  \label{eq:C005}
    \end{align}
 where we used the inequality $\frac {d^-}{dt}F(t)\geq -C$ by Lemma \ref{cor:802}.
On the other hand, by Lemma \ref{cor:802} again, we have
$$ \frac{d^-}{dt}F \geq - C \,O^{\al}Q^{2-\al} F.  $$
Therefore, for any small $\dd>0$ we have
 \begin{align*}
      \log F(-\delta)-\log F(-T)\geq -C \int_{-T}^{-\delta} O^{\al}Q^{2-\al}\,dt,
    \end{align*} where $C=C(\al, n)$ is a constant.
It follows that
    \begin{align*}
        \int_{-T}^{-\delta}\,O^{\al}Q^{2-\al}\,dt \geq \frac{1}{C} \log F(-T) -\frac{1}{C} \log F(-\delta)
        \geq \frac 1C \log \frac {F(-T)}{C}-\frac 1C\log \delta,
    \end{align*}
 where we used (\ref{eq:C005}).
 Suppose $\displaystyle \limsup_{t\ri 0}\,O^{\al}Q^{2-\al}|t|<A$.  Then  we have
\beq \frac 1C \log \frac {F(-T)}{C}-\frac 1C\log \dd\leq -A\log
\dd+A\log T.\label{eq708}\eeq
for every small $\delta$. It forces that $A \geq \frac{1}{C}$.
Replace $A$ by $\displaystyle \limsup_{t\ri 0}\,O^{\al}Q^{2-\al}|t| +\epsilon$ and let $\epsilon \to 0$.
Then we have
$$\limsup_{t\ri 0}\,O^{\al}Q^{2-\al}|t|\geq\frac 1C.$$
Therefore, (\ref{eqn:HG22_2}) is proved.
\end{proof}


\subsection{Further study}

 The methods and results in this paper can be generalized in the following ways.

 1. The method we developed in this paper reduces the convergence of the flow to three important steps: uniqueness of critical metrics, regularity improvement, and good behavior of some functional along the flow.
  Theorem~\ref{thmin:2} and Theorem~\ref{thmin:4} can be proved for minimizing extK metrics in a general class, assuming a uniqueness theorem of minimizing extK metric
  in a fixed $\mathcal{G}^{\C}$-leaf's $C^{\infty}$-closure, or a generalization of Theorem 1.3 of~\cite{[ChenSun]}.  This will be discussed in a subsequent paper.

 2. The fourth possibility of Donaldson's conjectural picture seems to be extremely difficult.
   By the example of G.Sz{\'e}kelyhidi, the flow singularity at time infinity could be very complicated.
   However, if we assume the underlying K\"ahler class to be $c_1(M,J)$ which has definite sign or zero,
   then the limit should be a normal variety and could be used to construct a destabilizing test configuration.

 3. Our deformation method  could be applied to a more general situation.  In section 3, we deformed the complex structures and the metrics within a given K\"ahler class.
   Actually, even the underlying K\"ahler classes can be deformed.   The general deformations will be discussed in a separate paper.

\vspace{0.5in}
Haozhao Li,  Department of Mathematics, University of Science and Technology
of China, Hefei, 230026, Anhui province, China and Wu Wen-Tsun Key
Laboratory of Mathematics, USTC, Chinese Academy of Sciences, Hefei
230026, Anhui,  China;  hzli@ustc.edu.cn.\\

Bing  Wang, Department of Mathematics, University of Wisconsin-Madison,
Madison, WI 53706, USA;  bwang@math.wisc.edu.\\

Kai Zheng,  Institut f\"ur Differentialgeometrie,
Leibniz Universit\"at Hannover, Welfengarten 1, 30167 Hannover,
Germany, zheng@math.uni-hannover.de. \\

\end{document}